\documentclass[12pt]{iopart}
\usepackage{amsthm}
\usepackage{amsfonts,amssymb}
\usepackage{color}
\usepackage{subfigure}
\usepackage{graphicx}
\usepackage{placeins}
\usepackage[margin=2cm]{geometry}
\newtheorem{theorem}{Theorem}
\newtheorem{lemma}[theorem]{Lemma}

\newtheorem{definition}[theorem]{Definition}

\newtheorem{remark}[theorem]{Remark}
\usepackage{algorithmic}

\usepackage{amsopn}

\DeclareMathOperator{\diff}{d\!}

\DeclareMathOperator{\Diff}{Diff}
\DeclareMathOperator{\Id}{Id}

\newcommand{\adm[1]}{$#1-$admissible}
\newcommand{\lip}{Lipschitz }

\newcommand{\R}{\mathbb{R}}

\def\o{\circ}

\def\p{\partial}

\def\exp{\operatorname{exp}}

\providecommand{\norm}[1]{\left\|#1\right\|}
\newcommand{\pp}[2]{\frac{\partial #1}{\partial #2}}

\newcommand{\bbR}{\mathbb{R}}
\newcommand{\cN}{\mathcal{N}}
\newcommand{\cG}{\mathcal{G}}

\newcommand{\cH}{\mathcal{H}}

\newcommand{\Pbb}{\mathbb{P}}
\newcommand{\CJC}[1]{#1}
\renewcommand{\SLC}[1]{#1}

\begin{document}

\title{Bayesian data assimilation in shape registration}
\author{C.J. Cotter$^1$, S.L. Cotter$^2$ and F.-X. Vialard$^3$}
\eads{\mailto{simon.cotter@manchester.ac.uk}} \address{$^1$Department of
  Aeronautics, Imperial College London, London SW7 2AZ United Kingdom\\
  $^2$School of Mathematics, University of Manchester, Manchester M13 9PL, United Kingdom\\$^3$Institute for Mathematical
  Sciences, Imperial College London, 53 Prince's Gate, London SW7 2PG,
  United Kingdom}
\begin{abstract}
  In this paper we apply a Bayesian framework to the problem of
  geodesic curve matching. Given a template curve, the geodesic
  equations provide a mapping from initial conditions for the
  conjugate momentum onto topologically equivalent shapes. Here, we
  aim to recover the well-defined posterior distribution on the
  initial momentum which gives rise to observed points on the target
  curve; this is achieved by explicitly including a reparameterisation
  in the formulation. Appropriate priors are chosen for the functions
  which together determine this field and the positions of the
  observation points, the initial momentum $p_0$ and the
  reparameterisation vector field $\nu$, informed by regularity
  results about the forward model. Having done this, we illustrate how
  Maximum Likelihood Estimators (MLEs) can be used to find regions of
  high posterior density, but also how we can apply recently developed
  \SLC{Markov chain Monte Carlo (MCMC)} methods on function spaces to
  characterise the whole of the posterior density. These illustrative
  examples also include scenarios where the posterior distribution is
  multimodal and irregular, leading us to the conclusion that
  knowledge of a state of global maximal posterior density does not
  always give us the whole picture, and full posterior sampling can
  give better quantification of likely states and the overall
  uncertainty inherent in the problem.
\end{abstract}
\maketitle{}

\section{Introduction}
Geodesics on shape space induced from diffeomorphisms are proving to
be a powerful tool in the field of computational anatomy
\cite{MiTrYo2003,MiTrYo2006}. Not only do they provide a notion of
distance between topologically equivalent shapes, they also provide a
linear characterisation of patches of shape space centred on a
template image using the initial conditions of the conjugate momentum
along a geodesic \cite{VaMiYoTr2004}. This permits the application of
linear statistical techniques (the adaptation of principal component
analysis known as principal geodesic analysis, for example
\cite{FlLuPiJo2004}) since the momentum inhabits a linear cotangent
space. Bayesian statistical techniques have been developed in this
area in recent years, with advances in methods of determining the most
likely template for a dataset as well as the most likely metric
\cite{AlKuTr2010,AlKuTr2008,Ma_etal2008}. There are many different
shape spaces: the shape of curves, landmarks, images, fibre bundles
\emph{etc.}, but there is a unifying paradigm since in all cases the
underlying velocity field that generates the diffeomorphisms that act
on the shapes evolves according to the EPDiff equation
\cite{HoRaTrYo2004}.

In this paper, we shall address the following inverse problem: given
noisy observations of points around a curve (which may have been
obtained from a segmentation algorithm or human input), what is the
initial momentum that generates this curve from a given template? We
concentrate on the parameterisation-independent problem, in which we
treat curves as being equivalent if they are related \emph{via}
reparameterisation. This is a challenging problem which has been
addressed in a few different ways \cite{Glaunes,VaGl2005}, but here we
shall make use of the approach of \cite{CoHo2010} in which a
reparameterisation variable is included explicitly, and the
commutation of the reparameterisation with the geodesic flow equations
is exploited. In the inverse problem nomenclature, this allows us to
define an observation operator that produces point values on the
curve, at a cost of introducing this extra variable.

\CJC{Our inverse problem can be summarised as follows. For a chosen
  metric, equations (\ref{e: u}-\ref{e:p dot}) together with
  condition (\ref{e:reparam}) describe the motion along the geodesic
  path in the space of unparameterised curves, written in terms of a
  specific curve parameterisation $q(s,t)$, and a conjugate momentum
  $p(s,t)$, where $s$ is the parameterisation variable around the
  curve, and $t$ is time, \emph{i.e.} the parameter along the
  geodesic. The initial momentum $p(s,0):=p_0(s)$ completely
  determines the direction of the geodesic, and thus the
  parameterisation-independent geodesic between two parameterised
  curves $q^1(s)$ and $q^2(s)$ can be completely described by the
  initial momentum $p_0(s)$ and a reparameterisation map $\eta(s)$
  which is identified through a one-dimensional vector field
  $\nu(s)$. Given a template curve $q^1$, $p_0$ completely describes
  the final curve $q^2$ and $\nu$ completely describes a particular
  parameterisation of $q^2$. We wish to identify the curve $q^2$ which
  is most consistent with a set of observed data points, and we do
  this by defining a map $\mathcal{G}(p_0,\nu)$ which returns a finite
  set of points on $q^2$. The inverse problem is to solve $y =
  \mathcal{G}(p_0,\nu) + \eta$ given finite data $y$ and
  observation noise $\eta$. Since the inputs to $\mathcal{G}$ are
  functions and the output is a finite-dimensional vector, the problem
  is underdetermined.}

\SLC{We choose to adopt the Bayesian approach in this paper, for two
  main reasons. Firstly, the inverse problem of finding the geodesic
  flow map which deforms one shape into another from observations of
  the second shape is ill-posed and underdetermined. Secondly, the
  observations that are made are noisy in nature, and as such a
  probabilistic setting in which we can quantify the uncertainty
  inherent in the problem is appropriate. As the unknown quantities
  that we would like to recover are functions, we must take extra care
  to ensure that the problem is formulated
  non-parametrically\cite{Stuart_Bayes}. As such, we look to follow
  the methodologies that have been utilised previously in various
  fluid mechanics problems \cite{cdrs08,cds10,cds11}.

  Once the Bayesian inverse problem has been defined in such a way as
  to be well-posed on function space, we must consider how appropriate
  algorithms can be implemented on a computer in order to characterise
  the posterior probability distributions on the unknown
  functions. The unknown functions are of course discretised within
  the algorithms, and because it may be desirable to refine the mesh
  to obtain further accuracy, it is important that the MCMC methods
  used to sample the posterior converge in distribution at rates
  independent of the discretisation used. A range of MCMC methods
  which are well-defined on function space (and as such have mixing
  rates which are independent of discretisation of those functions)
  are available\cite{crsw11}, and have been used for Bayesian inverse
  problems on function spaces\cite{cds10,cds11}. The properties of
  these algorithms make them a highly viable candidate for the
  analysis of the problems identified in this paper. Specifically, a
  function space analogue of the random walk Metropolis-Hastings
  (RWMH) algorithm will be used in the numerical examples in this
  paper, referred to in \cite{Stuart_Bayes,crsw11} as the pCN
  algorithm.  }


In Section \ref{sec:eq} we introduce the problem, and describe the
equations of motion of the shape as it is deformed by the velocity
field, and demonstrate how we can find geodesic paths in shape space
between two shapes. We also show in this section that the deformed
shape is Lipschitz continuous with respect to two functions, the
initial momentum $p_0$, and the reparameterisation function
$\nu$. \SLC{In order to tackle an inverse problem on function spaces,
  it is necessary to show that this problem is well-posed. This
  amounts to finding function spaces on which the forward problem is
  locally Lipschitz, and in Section \ref{sec:prop} we find sufficient
  conditions on the regularity of these spaces.} \SLC{In Section
\ref{sec:bayes} we frame the inverse problem as a well-posed Bayesian
inverse problem, and identify minimum regularity priors for the
functions, informed by the analysis in Section \ref{sec:prop}.}

In Section \ref{sec:MCMC} we describe the Markov chain Monte Carlo
(MCMC) algorithm that we use to sample directly from the well-defined
posterior density on {$p_0$ and $\nu$}. In Section \ref{sec:NA} we
briefly discuss how we numerically approximate the dynamics described
in Section \ref{sec:eq}, before presenting some illustrative numerical
examples in Section \ref{sec:num}. We finish in Section \ref{sec:conc}
with some conclusions and discussion.

\section{Description of the forward model}\label{sec:eq}

In this section, we review the equations of motion for curves in the
plane acted on by geodesics in the diffeomorphism group, and explain
how these equations can be used to provide a mapping from the space of
periodic scalar functions (which turn out to be the normal
component of a conjugate momentum variable) to topology-preserving
nonlinear deformations of some chosen curve in the plane. We then
provide some analytical results that are required for defining the
associated Bayesian inverse problem.

In our approach, we parameterise an oriented curve in the plane as a
continuous function $q$ from a space $S$ (such as the circle, $S^1$)
into $\mathbb{R}^2$ \emph{i.e.}, $q\in
C^0(S^1,\mathbb{R}^2)$. Although in this paper we concentrate on the
lower dimensional case of a curve in the plane, it may be extended to
surfaces in $\mathbb{R}^3$, which provides many applications in
medical imaging, for example. We parameterise an evolving curve as
$q(s,t)$, where $s\in S$ is the parameter around the curve, and $t\in
[0,1]$ is the time parameter. Following the methodology of
\cite{MiYo2001,GlTrYo04,VaGl2005} we constrain the motion of the curve
$q(s,t)$ to the action of diffeomorphisms by requiring that
\begin{equation}
\label{e:motion}
\pp{}{t}q(s,t) = u(q(s,t),t)
\end{equation}
where $u(x,t)$ is a time-parameterised family of vector fields on
$\mathbb{R}^2$. This guarantees that the topology of the curve is
preserved (\emph{i.e.} there are no overlaps or cavitations).  We wish
to study curve evolution from a template curve $\Gamma^1$
(parameterised by $q^1(s)$) to a target curve $\Gamma^2$
(parameterised by $q^2(s)$). However, since we only want information
about the shape of the curve, and not the specific parameterisation,
we do not wish to enforce that any specific point $q^1(s)$ gets mapped
to any specific point on $\Gamma^2$. We instead use the following
generalised boundary condition
\begin{equation}
\label{e:bcs}
q(s,0) = q^1(\eta(s)), \quad q(s,1) = q^2(s),
\end{equation}
for arbitrary reparameterisations $\eta \in \Diff_+(S)$, the
orientation-preserving subgroup of the diffeomorphism group $\Diff(S)$
on $S$. If the boundary conditions \eref{e:bcs} are satisfied, we say
that $u$ describes a \emph{path} between $\Gamma^1$ and $\Gamma^2$. We
select a function space $B$ for vector fields, and define the
\emph{distance} along the path as
\begin{equation}
\label{e:distance}
\int_0^1 \frac{1}{2}\|u\|^2_B \diff{t}.
\end{equation}
For simplicity we assume that $B$ is a Hilbert space and that there
exists an operator $A$ such that
\begin{equation}\label{def:A}
\|u\|^2_B = \langle u,Au\rangle_{L^2}.
\end{equation}
The \emph{shortest path} between $\Gamma^1$ and $\Gamma^2$ is defined
by minimising \eref{e:distance} over $u$ and $\eta$ subject to
\eref{e:motion} and the boundary conditions \eref{e:bcs}.  To obtain
the equations of motion for the shortest path, we introduce Lagrange
multipliers $p(s,t)$ (which we call the ``momentum'') to enforce
\eref{e:motion}, and seek extrema of the action
\[
\mathcal{A} = \int_0^1 \frac{1}{2}\|u\|^2_B + \langle p,\dot{q}-u(q)\rangle\diff{t}.
\]
This leads to the following equations of motion in weak form:
\begin{eqnarray}
\label{e: u}
\langle \delta u, A u\rangle_{L^2}
  - \langle p,\delta u(q)\rangle & = & 0, \\
\label{e:q dot}  \int_0^1\left\langle \delta p,\pp{q}{t}-u(q)
\right\rangle\diff{t} & = & 0, \\
\label{e:p dot}  \int_0^1\left\langle p, \pp{\delta q}{t} - 
(\nabla u(q))\delta q\right\rangle\diff{t}
& = & 0,
\end{eqnarray}
where $\delta p$ and $\delta q$ are space-time test functions, with
\[
u,\delta u\in B, \quad p,\delta p\in L^2, \quad
q,\delta q \in H^1.
\]
If the minimisation is taken over all reparameterisations $\eta \in \Diff_+(S)$
then we obtain the condition
\begin{equation}
\label{e:reparam}
p\cdot\pp{q}{s} = 0, \quad \forall t\in [0,1].
\end{equation}
The condition states that the momentum $p$ is normal to the shape.  As
described in \cite{CoHo2010}, if this condition is satisfied, then the
curve evolution is invariant under the transformation
\begin{equation}
\label{reparam}
(p,q) \mapsto 
(\bar{p},\bar{q}) = \left({p\circ \eta}{\pp{\eta}{s}} ,q\circ\eta\right)
\end{equation}
for $\eta\in\Diff_+(S)$, which is the cotangent-lift of the
transformation $q\mapsto q\circ\eta$. This means that if equations
(\ref{e: u}-\ref{e:p dot}) are solved with initial conditions $(p,q)$
and $(\bar{p},\bar{q})$ then $\bar{q}=q\circ\eta$ at time 1. As
described in \cite{GaBaRa2011}, solutions that satisfy \SLC{equation}
\ref{e:reparam} are parameterised realisations of geodesics on the
shape space $C^0(S^1,\mathbb{R}^2)/\Diff_+(S)$. 

We define the time-1 flow map $\Psi$
\[
\Psi (p|_{t=0},q|_{t=0}) = q|_{t=1},
\]
where $p$ and $q$ satisfy equations (\ref{e: u}-\ref{e:p dot}). Having
fixed $q|_{t=0}=q^1$, we define a parameterisation-independent mapping
between scalar functions $p_0:S\mapsto \mathbb{R}$ (the normal
component of the initial conditions of $p$) and topologically
equivalent curves $q^2$ obtained by
\[
q^2 = \Psi(p_0n,q^1),
\]
where $n$ is the normal vector to the curve $q^1$.  The power of this
mapping is that it allows us to perform linear operations on the space
of functions $p_0$, such as averages. 

In this paper, our aim is to estimate the probability measure for
$p_0$ given a set of observed points $(q_1,\ldots,q_n)$ from the curve
$q^2$. These points may have been obtained from a digitised medical
image, for example. These points are assumed to be sorted in order
according to the orientation of the curve. We cannot directly solve
the inverse problem of finding $p_0$ such that
\[
\Psi(p_0n,q^1)(s_i) = q_i, \qquad i=1,\ldots,n, 
\]
since we do not know the values $s_i$ of the curve parameter $s$
that correspond to each $q_i$. To fix this, we introduce a
reparameterisation variable $\eta\in \Diff_+(S^1)$, and seek
$(p_0,\eta)$ such that
\[
\Psi \left(
\left(p_0n\right)\circ\eta \,
\pp{\eta}{s}
,q^1\circ\eta\right)(s_i) = q_i, \quad i=1,\ldots,n,
\]
with $\{s_i\}_{i=1}^n$ is some chosen distinct sequence of points in
$S$ ordering according to orientation, \emph{e.g.} equispaced
points. This guarantees to preserve the ordering since the curve is
evolved by a diffeomorphism.  The introduction of the
reparameterisation variable does not alter the shape of the obtained
curve, just the particular parameterisation used.

{Following \cite{CoClPe2011}, we construct reparameterisations from
scalar periodic functions $\nu$ by solving
\begin{equation}
\label{e: lie}
  \pp{\chi}{t}(s,t) = \nu(\chi(s,t)), \quad \chi(s,0)=\Id, 
\quad \chi(0,t)=\chi(1,t), 0\leq t,
\end{equation}
where $\Id$ is the identity map ($\Id(s)=s$). The reparameterisation
map is then evaluated from $\eta(s)=\chi(s,1)$. The function $\nu$
represents the generating vector field for the reparameterisation
$\eta$ on the scalar interval $[0,1]$ with periodic boundary
conditions.} This is known as \emph{Lie exponentiation}\footnote{As
described in \cite{KeWe2008}, not all reparameterisations of the
circle can be obtained from Lie exponentiation, and a more general
approach would be to generate reparameterisations from geodesics in a
way that is similar to how the deformations of the curves are obtained
(known as Riemann exponentiation).}, and guarantees that $\eta$ is an
orientation preserving, smooth invertible map (for sufficiently smooth
$\nu$).  {This defines a reparameterisation map:
\[
\mathcal{R}(p,q,\nu) =  \left(
{\left(p_0n\right)\circ\eta}
{\pp{\eta}{s}}
,q^1\circ\eta\right)
\]}
where $\eta$ is obtained from (\ref{e: lie}).  An illustrative diagram
is given in Figure {\SLC{\ref{forward diag}}}.
\begin{figure}[htp] 
\centerline{
\includegraphics[width=14cm]{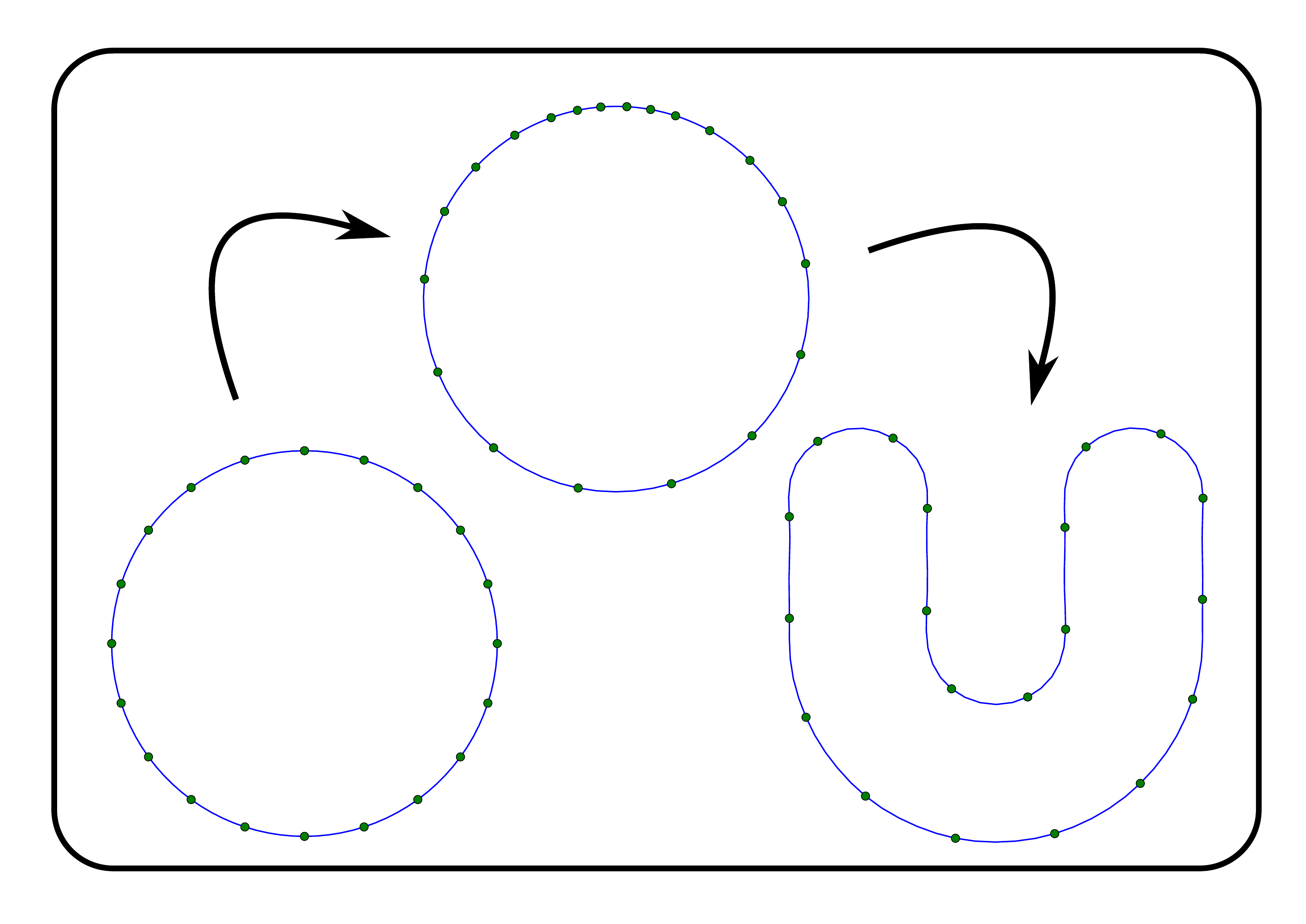}
}
\caption{\label{forward diag}{Diagram illustrating the forward
  model. Curves are shown with highlighted points to indicate the
  parameterisation. The reparameterisation velocity $\nu$ generates a
  reparameterisation of the curve (which also transforms the initial
  momentum $p_0$), and then the (transformed) initial momentum $p_0$
  generates an evolution of the curve by the time-1 flow map of
  equations (\ref{e: u}-\ref{e:p dot}). For the undiscretised
  equations changing $\nu$ but not $p_0$ leads to different
  parameterisations of the same curve, since reparameterisation and
  the time-1 flow map commute.}}
\end{figure} 

Hence we define the
observation operator $\mathcal{G}$ by
\SLC{
\begin{equation}
\label{e: observation}
\mathcal{G}(p_0,\nu) = 
\left(
\begin{array}{c c}
G(s_i)\\
\vdots \\
G(s_n) \\
\end{array}\right),
\end{equation}
where $G = \Psi \circ \mathcal{R}(p_0n,q^1,\nu)$.}
The normal component of $p_0$ then characterises the shape of the
target curve $\Gamma^2$ relative to the curve $\Gamma^1$, whilst the
generator variable $\nu$ merely describes the reparameterisation of the
target curve which is obtained at the minimum. 

\section{Properties  of the Observation Operator}\label{sec:prop}
In this section, we show that the observation operator $\mathcal{G}$,
given in Equation \eref{e: observation}, is Lipschitz in $p_0$ and
$\nu$, for appropriately chosen spaces. We make the following
definition.
\begin{definition}
  We say that a Hilbert space $B$ of vector fields is embedded in
  $C^n(\Omega,\R^d)$ if there exists a constant $C_e$ s.t. $\| v
  \|_{n,\infty} \leq C_e \|v\|_{B}$. In that case we also say for
  \SLC{notational} convenience that $B$ is \adm[n].
\end{definition}
We adopt the notation that $\nu \in D$. The main result of this
section is the following Theorem.
\begin{theorem}\label{Sh:Lip}
\begin{enumerate}
\item Let $B$ be \adm[2], and $D$ be
  \adm[1]. Then for $q_0 \in W^{1,\infty}$, the observation operator
  $\mathcal{G}$ is a Lipschitz continuous map from $(L^2,D)$ to
  $\mathbb{R}^{2n}$.
\item  Let $B$ be \adm[3] and $D$ be \adm[1]. Then for $p_0\in H^{-1}$ and
  $q_0 \in W^{1,\infty}$, the observation operator $\mathcal{G}$ is
  Lipschitz continuous.
\end{enumerate}
\end{theorem}

To show that $\mathcal{G}$ is Lipschitz, we first prove existence,
uniqueness and Lipschitz continuity for the maps $\Psi$ and
$\mathcal{R}$. This is then used to show that the observation operator
is Lipschitz continuous with respect to the normal component $p_0$ of
the initial conditions $p$ and the generator variable $\nu$. 

Let $B$ be the space of vector fields, $Q$ be a Hilbert space and $P$
be its dual. Here $Q$ represents a space of curves with a particular
norm, for example $Q = L^2(S_1,\R) $ or $Q= H^1(S_1,R)$. We want to
define the momentum map as given by
\begin{equation}
J : (p,q) \in P \times Q \mapsto J(p,q) :=\left( v \mapsto \left\langle p , v \o q \right\rangle_{P,Q} \right) \in B^* \,.
\end{equation}
{Here $B^*$ represents the dual space of $B$, the space of linear
  operators acting on $B$.}  In this aim, we assume in addition that
the map $B \times Q \mapsto Q$ defined by $(v,q) \to v \o q$ is a \lip
map on $Q$ for a fixed $v \in B$ namely there exists $C_c$ s.t.
\begin{equation} \label{LipCompAssump}
\| v \o q_1 - v \o q_2 \|_{Q} \leq C_c \| v \|_{B}  \|  q_1 -  q_2 \|_{Q} \,.
\end{equation}
As a consequence, since this composition map is linear on $B$, the map is \lip on $B \times Q$.
This assumption makes $J(p,q)$ well-defined and \lip on $P \times Q$ for the dual norm on $B^*$.

\begin{lemma} \label{MultiCases}
The assumption above is satisfied in the following situations
\begin{itemize}
\item $B$ is \adm[1] and $P = Q = L^2(S_1,\R)$.
\item $B$ is \adm[2], $P = H^{-1}(S_1,\R)$ and $Q = H^1(S_1,\R)$.
\end{itemize}
\end{lemma}

\begin{proof}
For both points, we use the integration formula for the path $q(t) = t q_1 - (1-t) q_0$
\begin{equation} \label{IntegralFormula}
f(q(1)) - f(q(0)) = \int_0^1 df(q(t)) (q_1 - q_0) \, dt \,,
\end{equation}
for $f \in C^1$.
For the first point, for any $f \in C^{1,\infty}(\R^d,\R^d)$,
\begin{equation} \label{LipComp}
\| f \o q_1 - f \o q_2 \|_{L^2} \leq \|f\|_{1,\infty} \|q_1 - q_2 \|_{L^2} \,,
\end{equation}
which gives the result.  For the second point, we note that
$H^1(S_1,\R)$ is Banach algebra, namely for any couple $(f,g) \in
H^1(S_1,\R)$ we have $$\|fg\|_{H^1} \leq \|f\|_{\infty}\|g\|_{H^1} +
\|g\|_{\infty}\|f\|_{H^1} \leq 4\|f\|_{H^1}\|g\|_{H^1}\,.$$ Therefore
in the Formula \eref{IntegralFormula}, the term $ df(q(t)) (q_1 -
q_0)$ can be understood as the product (generalised to $\R^d$) of two
elements in $H^1(S_1,\R^d)$ provided that $ df(q(t))$ belongs to
$H^1$.  Moreover $H^1$ is stable under the composition with a $C^1$
function bounded for
\begin{equation}
\| f \o q \|_{H^1} = \|f\|_{1, \infty} \|q\|_{H^1}
\end{equation}
 The second point then follows from 
\begin{equation} \label{SecondPoint}
 \| f \o q_1 - f \o q_2 \|_{H^1} \leq \| f \|_{2,\infty} \| q_1 - q_2 \|_{H^1} \,,
\end{equation}
\end{proof}

We write Equations (\ref{e: u}-\ref{e:p dot}) as follows:
\begin{eqnarray}
\dot{q} &= &u \o q \label{composition} \\
\dot{p} &= &-\nabla u(q) . p \label{multiplication} \\
u&=&J(p,q)^\sharp\,, \label{mom map}  
\end{eqnarray} 
where $J(p,q)^\sharp$ denotes the corresponding element in $B$ given by the Riesz theorem, i.e.
$$ \langle J(p,q)^\sharp, v \rangle_{B} = J(p,q)(v) \,.$$
Equation \eref{multiplication} is a pointwise matrix multiplication when $p$ is smooth. When $p \in H^{-1}$ for instance, the operation  is defined by duality,
defined as follows.
\begin{definition}
Let $Q$ be a Banach algebra then the multiplication $Q^* \times Q \mapsto Q$ is defined by duality
\begin{equation}
\langle q_1. \,p , q_2 \rangle = \langle p , q_1 q_2 \rangle
\end{equation}
\end{definition}

\begin{remark} \label{LipMultiplication}
The multiplication on a Banach algebra is by definition bilinear continuous. As a consequence, the multiplication inherits the same properties and in particular, the \lip property.
\end{remark}

We now prove that these shooting equations are a simple ODE in
infinite dimensions i.e. $\dot{X} = F(X)$ with state variable $X$ in a
Banach space; This function $F$ satisfies the classical \lip
assumption which means that solutions can be obtained using Picard\rq{}s
fixed point method; it also ensures that the solutions are \lip functions
of the initial conditions.
\begin{lemma}
System (\ref{composition}-\ref{mom map}) is an ODE on $P \times Q$ which satisfies the \lip condition for the two following cases \SLC{
on bounded balls in $P \times Q$:}
\begin{itemize}
\item $B$ is \adm[2], $P=L^2$ and $Q = L^{\infty}$
\item $B$ is \adm[3], $P=H^{-1}$ and $Q = H^1$\,.
\end{itemize}
\end{lemma}

\begin{remark}
The space $Q= L^{\infty}(S_1)$ is not a Hilbert space in this case but only a Banach space.
\end{remark}
\begin{proof}
  In the two claimed cases, we have by Proposition \ref{MultiCases}
  that $(p,q) \to J(p,q)^\sharp$ is \lip for the norm on $B$. Then
  using the inequality \eref{LipComp}, we deduce that $(p,q) \to
  J(p,q)^\sharp \o q$ is \lip as the composition of two \lip
  functions. Thus, we have treated the first component
  \eref{composition} of the ODE.

  Using Remark \ref{LipMultiplication}, it is then sufficient to prove
  that $(p,q) \in P \times Q \mapsto d[J(p,q)^\sharp](q) \in Q$ is a
  \lip mapping. Denoting $v_i = J(p_i,q_i)^\sharp$ for $i=1,2$, use of
  the triangle inequality gives
\begin{eqnarray}
 \|dv_1 \o q_1 - dv_2 \o q_2 \|_Q &\leq &C_e\|v_1 \|_{B} \|q_1 - q_2\|_Q +  \|dv_1 \o q_2 - dv_2 \o q_2 \|_Q \nonumber \\
& \leq& C_e\|v_1 \|_{B} \|q_1 - q_2\|_Q +  C_e\|v_1-v_2\|_{B} \|q_2\|_{q}\,, \label{LipIneq}
\end{eqnarray}
where we rely on the inequality \eref{SecondPoint} for the $H^1$ case and \eref{LipComp} for the first case.

Up to this point, we only have proven a local \lip condition. It
remains to prove that for $(p_0,q_0) \in B(0,r) \subset P \times Q$
with $r$ a positive real number, there exists a constant $M_{r,t}$
(the subindex indicating that this constant only depends on $r$ and
the time $t$) such that for all time $t\geq0$ the system
(\ref{composition}-\ref{mom map}) is \lip of \lip constant $M_r$. In
particular this implies that the solutions are defined for all
times. To this end, we remark that the geodesic equations
(\ref{composition}-\ref{mom map}) are variational with the Hamiltonian
is constant in time. This can be checked by a straightforward
calculation. Since the Hamiltonian reads $H(p,q) =
\norm{J(p,q)^\sharp}_B^2$, the \lip inequality \eref{LipIneq} becomes
\begin{equation}
 \|dv_1 \o q_1 - dv_2 \o q_2 \|_Q \leq 3r\, C_e\, \|q_1 - q_2\|_Q \,.
\end{equation}
For the equation \eref{composition}, the inequality \eref{LipCompAssump} implies
\begin{equation}
\norm{v_1 \circ q_1 - v_2 \circ q_2} \leq r C_c \norm{q_1-q_2}_{Q} + \norm{v_1 - v_2}_B \max{(\norm{q_1}_Q,\norm{q_2}_Q)}
\end{equation}
but by direct integration of Inequality \eref{LipCompAssump} we see 
that $\max{(\norm{q_1}_Q,\norm{q_2}_Q)} \leq C_c \,r\, t$ for any time
$t>0$.  A similar inequality holds for the momentum $p$, hence the
geodesic equations (\ref{composition}-\ref{mom map}) are \lip for a
constant $M_{r,t}$ that has polynomial growth in $t$ and in $r$.
\end{proof}

We now claim the following result which is based on Gr\"{o}nwall's lemma.
\begin{lemma}\label{LipDepShooting}
  Solutions of the geodesic equations (\ref{composition}-\ref{mom
    map}) have a \lip dependency w.r.t. their initial conditions for
  their respective norms.
\end{lemma}
We are now in a position to prove Theorem \ref{Sh:Lip}.
\begin{proof}
The observation operator $\mathcal{G}$ is the evaluation at points $(\eta(s_i))_{i=1}^n$ of the deformed template $q(1)$ obtained by the geodesic equations (\ref{composition}-\ref{mom map}) without reparameterisation. In other words,
\begin{equation} \label{EvaluationOperator}
\textrm{Ev}(p_0,q_0,\nu) =  
\left(
\begin{array}{c}
q(1) \circ \eta (s_1)  \\
\vdots \\
q(1) \circ \eta (s_n)  \\
\end{array}\right)
\end{equation}
where $q(1)$ is the solution at time $1$ of the geodesic equations
(\ref{composition}-\ref{mom map}) for initial data $p_0,q_0$ and
$\eta$ is the Lie exponential of $\nu$.  To prove that the observation
operator is \lip in those three variables, we observe that it is the
composition of the flow of the geodesic equations
(\ref{composition}-\ref{mom map}) and the right composition with the
reparameterisation by $\eta$.  The first operation is \lip continuous
by Corollary \eref{LipDepShooting}. Now the right composition is also
\lip continuous since $q(1) \in W^{1,\infty}$ (indeed since $B$ is
\adm[2] or \adm[3], the flow of the geodesic equations
(\ref{composition}-\ref{mom map}) preserves $W^{1,\infty}$) we have
\begin{equation}
\| q(1) \o \eta_1 - q(1) \o \eta_2 \|_{\infty} \leq \| \nabla q(1) \|_{\infty} \| \eta_1 - \eta_2 \|_{\infty} \leq  M\|  \nabla q(1) \|_{\infty} \| \nu_1 - \nu_2 \|_{D} \,.
\end{equation}
\end{proof}

\section{Bayesian Inversion}\label{sec:bayes}
\SLC{ In this section, we aim to frame the ill-posed inverse problem
  of finding $(p_0,\nu)$ given noisy observations of the target shape,
  as a well-posed Bayesian inverse problem on function space. The
  ``solution'' to this inverse problem is then given by a probability
  distribution. Bayes' formula allows us to construct, through
  combining prior beliefs about the functions, with observations of
  the system, the \emph{posterior} distribution of the possible values
  that the unknown functions could have taken.

\subsection{Prior Distributions}

  Extra care must be taken in this context, since the unknowns are
  functions, and as such the problem must be formulated
  non-parametrically. Following the philosophy of \cite{Stuart_Bayes},
  we use the results from Section \ref{sec:prop} to identify minimum
  regularity priors which will ensure that the problem is
  well-defined. For problems on infinite dimensional spaces, we must
  identify prior distributions for the functions, which have full
  measure with respect to function spaces on which the observation
  operator is Lipschitz continuous. Theorem \ref{Sh:collated}
  summarises these results for the problem presented in this paper,
  and follows at the end of this section.

  The role of the prior in this respect is two-fold. The prior can
  contain prior information or beliefs about the form or structure of
  the solution, possibly taken from previous observations. However
  another important role of the prior is to regularise the problem, to
  make it well-posed. Indeed it plays a very similar role to the
  penalty term in a Tikhonov regularisation, an optimisation
  formulation of the inverse problem. We will explore these
  similarities in Section \ref{sec:RWMH}, and exploit them in our
  numerical method.

\subsection{Observations and the Likelihood}

Uncertainty in this problem arises from errors incurred during
observation. Observational noise is often modelled as being additive
Gaussian. This uncertainty in the exact shape of the curve leads to
uncertainty in the value of the pair of functions which would deform
the reference shape to the observed shape. Understanding and
quantifying that uncertainty is a key part of solving this problem, as
naive least squares matching to the data would lead to an ill-posed
problem.

We assume that observations $y$ of the quantity of interest are noisy
in nature, satisfying:
\[
y = \cG(p_0,\nu) + \xi, \qquad \xi \sim \cN(0,\Sigma),
\]
where $\Sigma$ is assumed to be known, and where $\mathcal{G}$ is
defined to be the observation operator given by \SLC{\eref{e: observation}}.
That is, the function which returns to us what noiseless observations
we would make if we were to deform the template shape and
reparameterise using the function pair $(p_0,\nu)$.
This assumption allows us to compute the
likelihood that $y$ was observed with a given $(p_0,\nu)$, through the
probability density function of $\xi$:
\begin{equation}\label{eq:like}
\Pbb(y|p_0,\nu) \propto \exp\left ( -\frac{1}{2}\|\cG(p_0,\nu) -
  y\|^2_\Sigma \right ),
\end{equation}
where $\| x \|^2_\Sigma:= x^T\Sigma^{-1}x$ is the covariance weighted
norm.

\subsection{Bayes' Theorem and the Posterior}

Bayesian inversion is one way to regularise an ill-posed inverse
problem, and convert this into a well-posed one. Bayes' formula is
central to this. Prior beliefs about the quantities of interest are
blended with data from observations to give the posterior
$\mathbb{P}(u|y)$ distribution. The finite dimensional version of this
formula is
\begin{equation*}
\mathbb{P}(u|y) \propto \mathbb{P}(y|u)\mathbb{P}(u),
\end{equation*}
where $u$ denotes the unknown quantity that we wish to characterise
from the data, $\mathbb{P}(u)$ the probability distribution describing
our prior beliefs about those quantities, and
$\mathbb{P}(y|u)$ the likelihood function, which returns the
likelihood that we would make the observations $y$ given a
particular value of the quantity $u$.

{However, in this paper the quantity of interest is a pair of
  functions $u = (p_0,\nu)$.  In this infinite dimensional setting, we
  must use the analogous result regarding the Radon-Nikodym derivative
  of the posterior distribution with measure $\mu$, with respect to
  the prior distribution on $(p_0,\nu)$ with measure $\mu_0$:
\begin{eqnarray}
\label{Sh:RadNik}
  \frac{d\mu(p_0,\nu)}{d\mu_0} \propto \mathbb{P}(y|(p_0,\nu)),
\end{eqnarray}
where the likelihood expression is given by \eref{eq:like}.}

Note that this formula only holds if the posterior is indeed
absolutely continuous with respect to the prior distribution,
otherwise no such derivative is admitted. Sufficient conditions on the
priors for this to be satisfied are given in the following section.


}

\subsection{Prior and Posterior Distributions on
  $(p_0,\nu)$}\label{sec:prior}

\SLC{Before stating the main result of this Section, we must first
  introduce some notation. Let us consider the Helmholtz operator
with periodic boundary conditions

\[
\cH = I -\ell^2\pp{^2}{s^2},
\]
where $\ell \in \mathbb{R}$ defines a \emph{characteristic length
  scale}. Note also that for samples $u \sim \mathcal{N}(0,\delta
\mathcal{H}^{-\alpha})$ for some $\alpha > \beta + \frac{d}{2}$ for
some $\beta \ge 0$, then $u \in H^\beta$ almost
surely (follows from \cite{Stuart_Bayes}). The length scale parameter $\ell \in
\bbR^+$ allows us to control at which scale the smoothing properties
of the Laplacian take effect. With a larger value of $\ell$, a larger
value of $|k|$ is required before the effect of the Laplacian becomes
dominant. Likewise, as $\ell \to 0$, $\cH \to -\Delta$, meaning that
the Laplacian is dominant on all scales. The choice of this value,
however, does not affect the overall regularity of samples drawn from
the distribution $\cN(0,\cH^{-\alpha})$.

\begin{theorem}\label{Sh:collated}
  Let $B$ be 3-admissable, $D$ be 1-admissable, and $\mu_0(p_0,\nu) =
  \mathcal{N}(0,\delta_1\cH_1^{-\alpha_1})\times
  \mathcal{N}(0,\delta_2\cH_2^{-\alpha_2})$ for
  $\alpha_1>-\frac{1}{2}$, $\alpha_2>3/2$, $\delta_1,\delta_2 \neq 0$,
  and where $\cH_i = (\ell_iI - \Delta)$ where $\ell_i \in \bbR^+$ for
  $i\in \{1,2\}$. Then $\cG$ is measurable with respect to $\mu_0$,
  and the posterior measure $\mu$ is absolutely continuous with
  respect to $\mu_0$, with Radon-Nikodym derivative given by
  (\ref{Sh:RadNik}).
\end{theorem}
\begin{proof}
  Result follows by the Sobolev embedding theorem, results in
  \cite{Stuart_Bayes}, and Theorem \ref{Sh:Lip}.
\end{proof}

For the posterior distribution to be well-defined, it must be
absolutely continuous with respect to the choice of prior distribution
on the two unknown functions of interest. This choice of prior
distribution is informed by the properties of the forward model, and
the observation operator $\mathcal{G}$. Theorem \ref{Sh:collated}
gives us Gaussian priors which have sufficient regularity to ensure
that this condition is met.}

\section{Characterising the Posterior Density}\label{sec:MCMC}
In this section we introduce the statistical methods which we use in
order to characterise the posterior distribution of the function pair
$(p_0,\nu)$ given a choice of prior distribution and a set of noisy
observations of the target shape. Monte Carlo Markov chains (MCMC) are
a set of tools which allow us to sample from a desired probability
distribution. A chain of states is formed, which converges in
distribution to the target density. {We then describe a method for
initialising the MCMC method efficiently.}

\subsection{The Function Space Random Walk Metropolis-Hastings
  algorithm (pCN)}\label{sec:RWMH}
For the purpose of the numerics, we use a version of the Random Walk
Metropolis Hastings (RWMH) algorithm framed on function spaces,
\SLC{referred to as the pCN algorithm\cite{Stuart_Bayes,crsw11}.} This
method is superior to the usual vanilla random walk method in that it
is well-defined on function spaces, and as such is robust under
refinement of any discretisation of the forward model
\cite{crsw11}. This means that the rate of convergence of statistics
in the Markov chains are independent of this discretisation, whereas
the Metropolis Hasting algorithm with vanilla random walk proposal
degenerates as the forward model's discretisation is refined. The
following describes the pCN sampler, where $\Phi =
\frac{1}{2}\|\cG(p_0,\nu)-y\|^2_\Sigma$:
\begin{algorithmic}[1]
  \STATE $u_0 = (p_0,\nu)_0$ \FOR{$i=1:N$} \STATE Sample $v =
  (1-\beta^2)^{1/2}u_{i-1} + \beta w$, where $w \sim \mu_0((p_0,\nu))$
  \STATE $a(u_{i-1},v) = \min \left \{ 1, \exp(\Phi(u_{i-1}) -
    \Phi(v)\right \}$ \STATE Sample $u \sim U([0,1])$ \IF{$u <
    a(u_{i-1},v)$} \STATE $u_i = v$ \ELSE \STATE $u_i = u_{i-1}$
  \ENDIF
  \ENDFOR
\end{algorithmic}
\SLC{
The acceptance probability in this algorithm is independent of the
approximation method used for the unknown functions, and as such has
superior mixing properties as compared with standard methods for finer
discretisations. Further to this basic algorithm, we can alter this to also
use the data to infer on the variance of the observational noise, if
we no longer assume that this is known\cite{crsw11}; we make use of
this in the numerical experiments in Sections \ref{sec:PC} and
\ref{sec:par}.}

\SLC{
MCMC methods usually need a period of ``burn-in'' in which we try to
ensure that the chain starts in a region which is not in the tails of
the posterior distribution. MAP estimators can be used to try to find
regions of high posterior density, as described in \cite{cdrs08}. For
the problem as we have defined it, this equates to finding the
solution to:}
\[
\min_{p_0 \in L^2, \nu\in H^2} L(p_0,\nu) = \min_{p_0 \in L^2, \nu\in H^2} \frac{1}{2}\|\cG(p_0,\nu) -
y\|^2_\Sigma + \frac{1}{2}\|(p_0,\nu)\|^2_{\mu_0}
\] 
where $\|(\cdot,\cdot)\|_{\mu_0}$ is the equivalent penalty term to
that induced in the posterior measure by the choice of prior measure
$\mu_0$, or the \emph{Cameron-Martin} norm corresponding to
$\mu_0$. In particular, if (as is the case for our purposes)
$\mu_0(p_0,\nu) = \mathcal{N}(0,\delta_1\cH^{-\alpha_1})\times
\mathcal{N}(0,\delta_2\cH^{-\alpha_2})$, then
\begin{eqnarray*}
\|(p_0,\nu)\|^2_{\mu_0} &=& \sum_k \delta_1 |p_k|^2 (\ell_1 - |k|^2)^{\alpha_1}
+ \delta_2 |\nu_k|^2 (\ell_2 - |k|^2)^{\alpha_2}
\end{eqnarray*}


Various descent methods can be used to try to find local
minima of this quantity. Most of these methods incorporate gradient
information of some type to attempt to search for the minima in
appropriate directions. These methods include steepest descent and
conjugate gradient among others.

The method that we will utilise in our numerics in order to initialise
our Markov chains is called the Broyden-Fletcher-Goldfarb-Shanno
(BFGS) method \cite{broyden70,shanno85}. This method uses the gradient
information from the last two states in the chain to approximate a
Hessian matrix, which is then used to choose an appropriate direction
in which to search for the local minimum.

{Once the BFGS method has converged within certain tolerances, the
  state that has been settled on can be used as an initial state in
  the MCMC chain. Assuming that the BFGS method has converged close to
  a global or even local maxima of probability density, then the chain
  will more quickly enter probabilistic equilibrium. For example,
  consider the data used in Section \ref{sec:par}. \SLC{The ratio of
  probability densities of the state $(p_0,\nu) = (0,0)$ with the
  state the BFGS method converged to was $\exp(-958.85)$, a huge
  difference.} 
}

\section{Numerical Approximation of $\cG$}\label{sec:NA}
The observation operator consists of the concatenation of four maps:
1) the generation of the reparameterisation variable $\eta$ from the
vector field $\nu$ according to \eref{e: lie}, 2) the application of
the reparameterisation formula according to \eref{reparam}, 3) the
time-1 flow map of the geodesic equations according to (\ref{e:
  u}-\ref{e:p dot}), and 4) the evaluation of the curve $q|_{t=1}$ at
the finitely chosen points. Stages (1) and (3) are both numerically
discretised using particle-mesh methods as described in \cite{Co2008}.
These methods allow efficient evaluation of the velocity field at
particle locations whilst preserving variational structure and making
it possible to compute the exact adjoint equations of the discrete
system which are used in the implementation of the deterministic
burn-in. Stages (2) and (4) use cubic B-spline interpolation, which
also make computing the adjoint tractable. Details of the numerical
approximation procedure, along with numerical tests that show that the
numerical scheme is second-order convergent in space, are given in
\cite{CoClPe2011}. {In summary, the discretisation replaces the
  parameterised curve $q(s)$ as an ordered discrete set of points
  $\{q_i\}_{i=1}^{n_p}$, similarly the conjugate momentum $p(s)$ is
  replaced by a set ${p_i}_{i=1}^{n_p}$ of (co)-vectors associated
  with each point. Convergence is achieved as $n_p\to\infty$, assuming
  a sufficiently smooth velocity field; the convergence rate was
  demonstrated to be second-order in \cite{CoClPe2011}. The velocity
  field $u$ is discretised on a static, uniform mesh with $n_g\times
  n_g$ grid points. In the numerical algorithm, the momentum is
  interpolated from the points to the grid using cubic B-splines; the
  elliptic operator A is then approximately inverted on the grid using
  discrete Fourier transform, and the resulting velocity is
  interpolated back to the points $q_i$. As described in
  \cite{Co2008}, this interpolation preserves the Hamiltonian
  structure of the equations, and as described in \cite{CoClPe2011},
  allows an exact adjoint of the numerical model to be constructed,
  which is used to obtain derivatives of the observation
  operator. Convergence was also demonstrated to be second-order as
  $n_g\to\infty$ in \cite{CoClPe2011}. Alternatively, one can select a
  finite value for $n_g$ whilst keeping $q$ and $p$ continuous; this
  just leads to a modification of the operator $A$. The discretisation
  for the reparameterisation variable $\nu$ was introduced in
  \cite{CoClPe2011}. The reparameterisation velocity $\nu(s)$ is also
  discretised as a discrete set of values $\{\nu_i\}_{i=1}^{n_p}$, and
  the reparameterisation map is obtained by interpolating $\nu$ to the
  reparameterised locations; $q$ and $p$ are then pulled back under
  the reparameterisation map by interpolation. As discussed in
  \cite{CoHo2010}, reparameterisation and flow along geodesics are
  commuting operations; in \cite{CoClPe2011} the commutation error was
  shown to converge to zero at second-order as $n_p\to\infty$.}


\section{Numerical Results}\label{sec:num}
We present numerical results to show that the algorithm is
successfully drawing samples from the posterior distribution, and then
go on to study the posterior distribution in certain data
scenarios. The first thing to address is our choice of template shape
$\Gamma^1$ and the parameterisation $q^1(s)$ that we use for this
shape. Since we are only considering closed curves, it seems natural
to choose a circle for this to keep things as simple as possible. We
also wish to choose a nice smooth parameterisation for this shape,
centred in the middle of our domain $\mathbb{T}^2 = [0,2\pi)^2$ so we
pick
\begin{equation}\label{template}
  q^1(s) = (\cos(s) + \pi ,\sin(s) + \pi), \qquad s\in [0,2\pi).
\end{equation}
We will engage in simulation studies in which the data is itself
produced by employing the numerical simulation of a forward PDE
model. In all the numerics, we assume that the noise $\xi$ through
which we make the observations
\[
y = \cG(p_0,\nu) + \xi, \qquad \xi \sim \cN(0,\Sigma),
\]
has a diagonal covariance matrix $\Sigma$, which we describe for each
experiment.

In terms of the approximation of the forward model in the algorithm,
50 time steps are used in which to deform the parameterisation of the
shape from time $t=0$ to $t=1$. The curves themselves are approximated
by 100 points, and with a $64\times 64$ grid approximating the
underlying velocity field. The values of $s_j$ at which observations
are made are given by $\left \{\frac{2\pi s_j}{N} \right
\}_{j=0}^{N-1}$. These parameters are used in the model to create the
data, and in the implementation of the statistical algorithm.

The prior distributions on $p_0$ and $\nu$ are
$\mathcal{N}(0,\delta_1\cH^{-\alpha_1})$ and
$\mathcal{N}(0,\delta_1\cH^{-\alpha_1})$ respectively, with $\alpha_1
= 0.55$, $\alpha_2 = 1.55$, $\delta_1 = 30$ and $\delta_2 = 0.05$. Note
that these parameters are sufficient to ensure that Corollary
\ref{Sh:collated} holds.

In the following sections we plot normalised histograms for particular
Fourier modes of the two functions of interest, from converged Markov
chains which are invariant with respect to the posterior distribution
as described in Section \ref{sec:MCMC}.

\subsection{Posterior Consistency}\label{sec:PC}
It seems reasonable to expect that as we increase the amount of
informative data that we are using in our inference, the closer our
posterior mean will be to the functions that created the data, and
that at the same time the uncertainty in that estimation will
decrease. \SLC{Posterior
  consistency is not yet a developed topic in inverse problems, with
  most theory developed only in the Gaussian (and often linear)
  cases\cite{Agapiou_post_con,Knapik_BIPWGP,knapik_heat_eq}. As such
  we}{ do not aim to prove posterior consistency in this
  context, but instead aim to present results which appear to indicate that
  this is indeed the case in this application.} In this set of
numerical experiments, we take a draw $(p_0,\nu)$ from the prior
measure, and using our approximation of the forward model, create data
$y$ such that
\[
y = \left \{ q^2\left ( \frac{2\pi n}{N} \right ) \right
\}_{n=0}^{N-1} + \xi, \qquad \xi \sim \cN(0,\Sigma), 
\]
with increasing $N$.

In this experiment, the data is simulated using the numerical
approximation of the forward model, {but in order to avoid
  committing an \emph{inverse crime}\cite{JKES} we use a much higher
  resolution of the curve when simulating the data. 1000 points were
  used to describe the curve in the data creation routines, while just
  100 points were used within the approximation used for the MCMC
  algorithm.} The initial momentum and reparameterisation functions
that created the data were drawn from the prior
distributions. Mean-zero Gaussian noise with covariance matrix $\Sigma
= \sigma^2 I$ with $\sigma = 10^{-2}$ was added, {but was not
  assumed to be known. In the results that follow, the MCMC method
  described in Section \ref{sec:RWMH} was used to sample from the
  joint distribution on the functions of interest $(p_0,\nu)$, and the
  observation noise variance $\sigma^2$.

\begin{figure}[htp] 
\begin{center} 
\mbox{ 
\subfigure[]{\scalebox{0.31}{\includegraphics{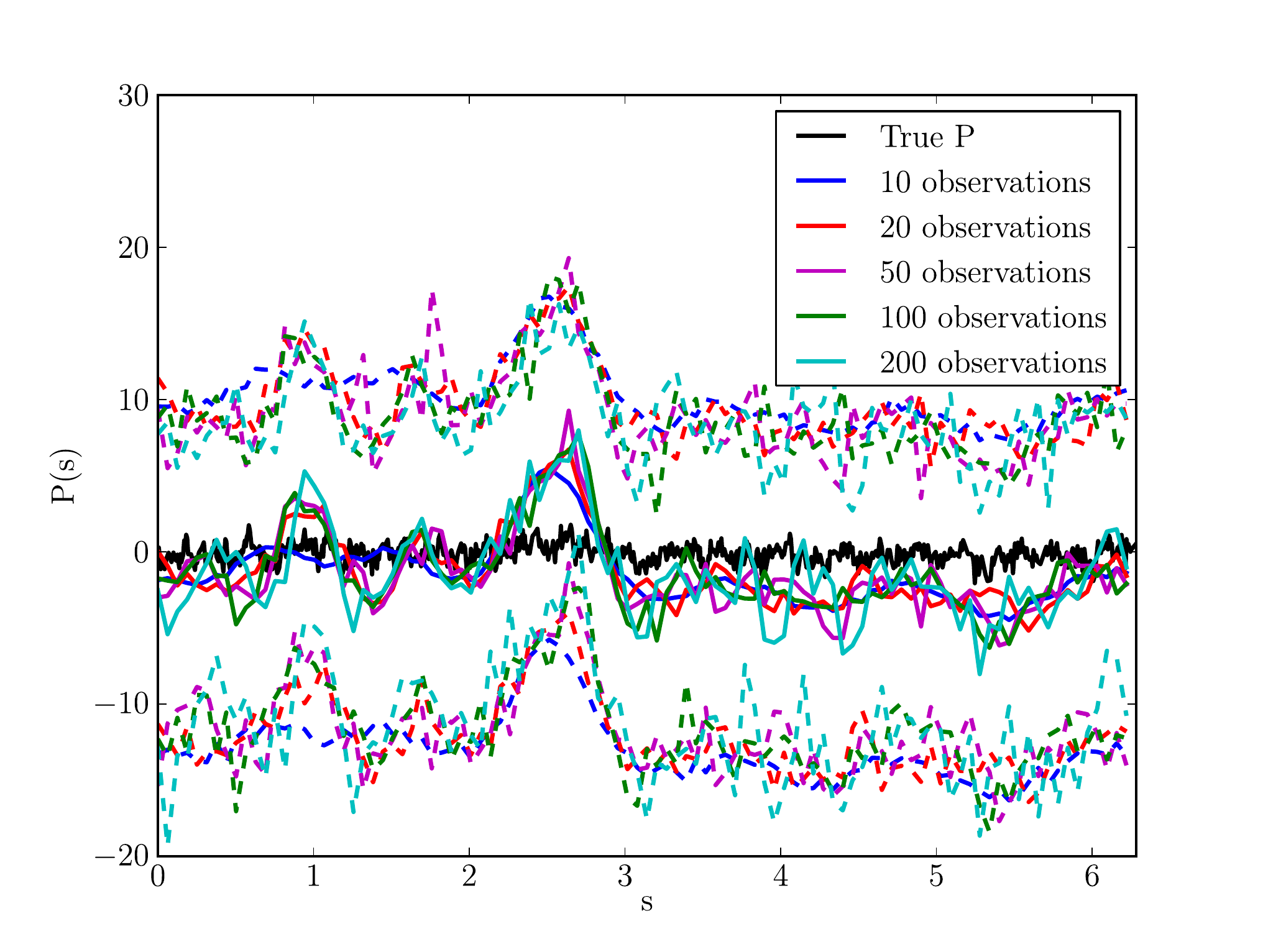}}} 
\subfigure[]{\scalebox{0.31}{\includegraphics{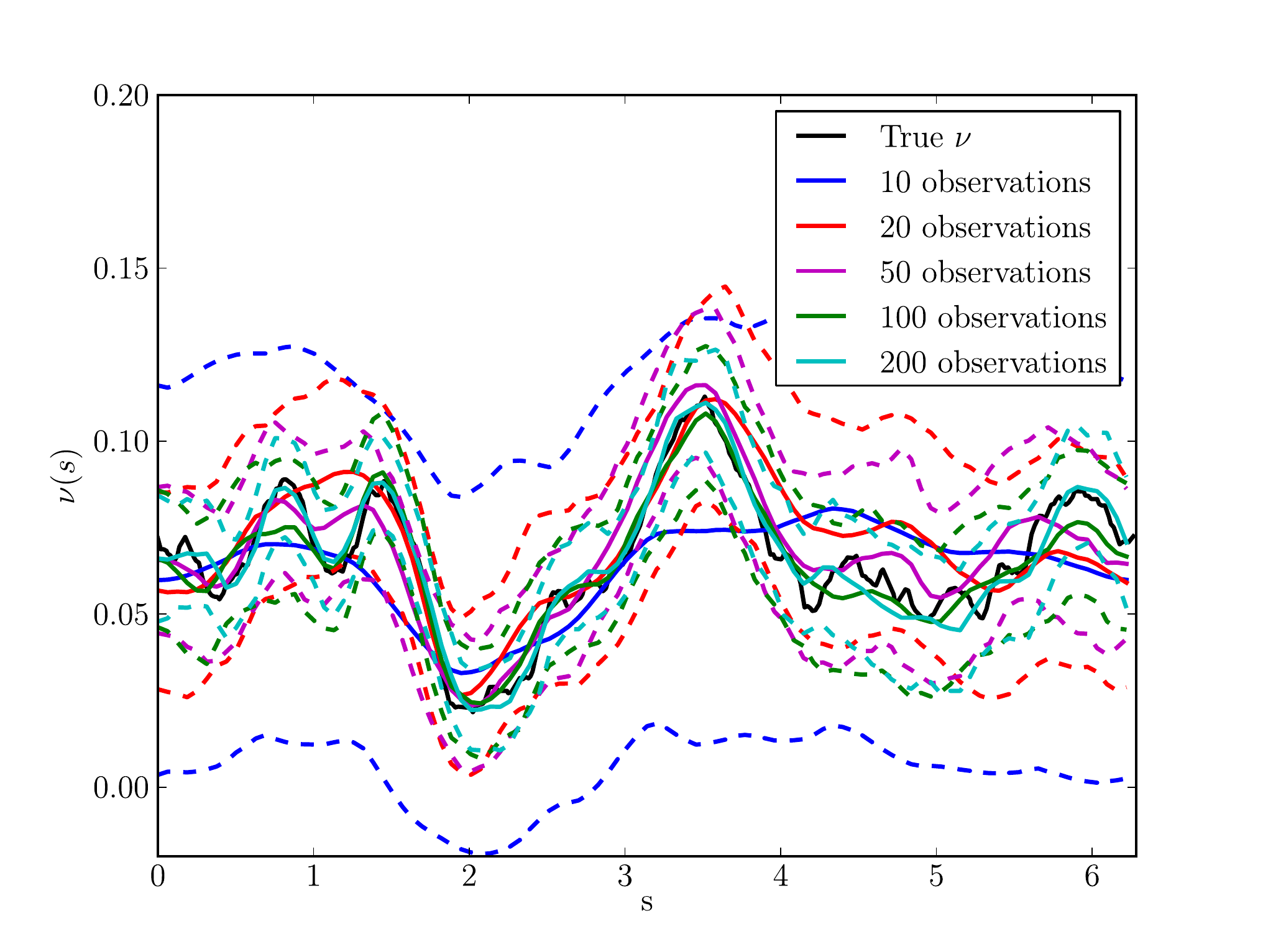}}} 
} 
\caption{\it Plots of the mean values (solid lines) of (a) $p_0$ and
  $\nu$ occurring in the posterior distribution using data with
  increasing number of points, with observational noise covariance
  equal to $\sigma I$ with $\sigma=0.01$. The dashed lines denote $\pm
  3$ standard deviations (approximately 99.7\% confidence
  interval)\label{NObs:means}}
\end{center} 
\end{figure}

Figure \ref{NObs:means} shows plots representing the mean and variance
of the approximation of the functions $p_0$ (left) and $\nu$ (right)
given an increasing number of observations. As the number of
observations is increased, the better we recover the actual value of
$\nu$ that was used to create the data, and the level of uncertainty
in the distribution is reduced. The convergence of $p_0$ is not so
clear, as the convergence is not expected to be pointwise in this
case, but in $L^2(S^1)$.

\begin{figure}[htp] 
\begin{center} 
\mbox{ 
\subfigure[]{\scalebox{0.31}{\includegraphics{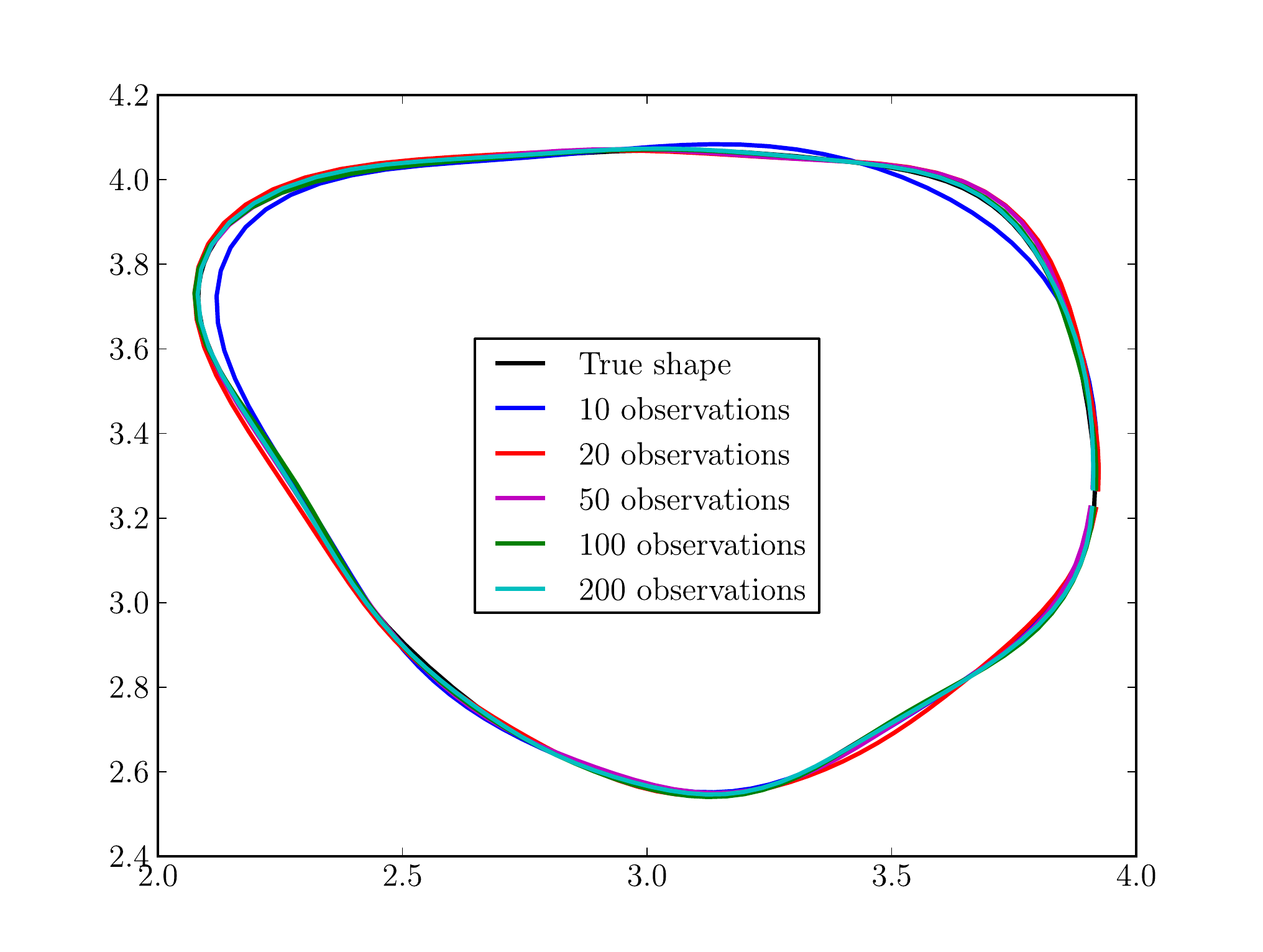}}} 
\subfigure[]{\scalebox{0.31}{\includegraphics{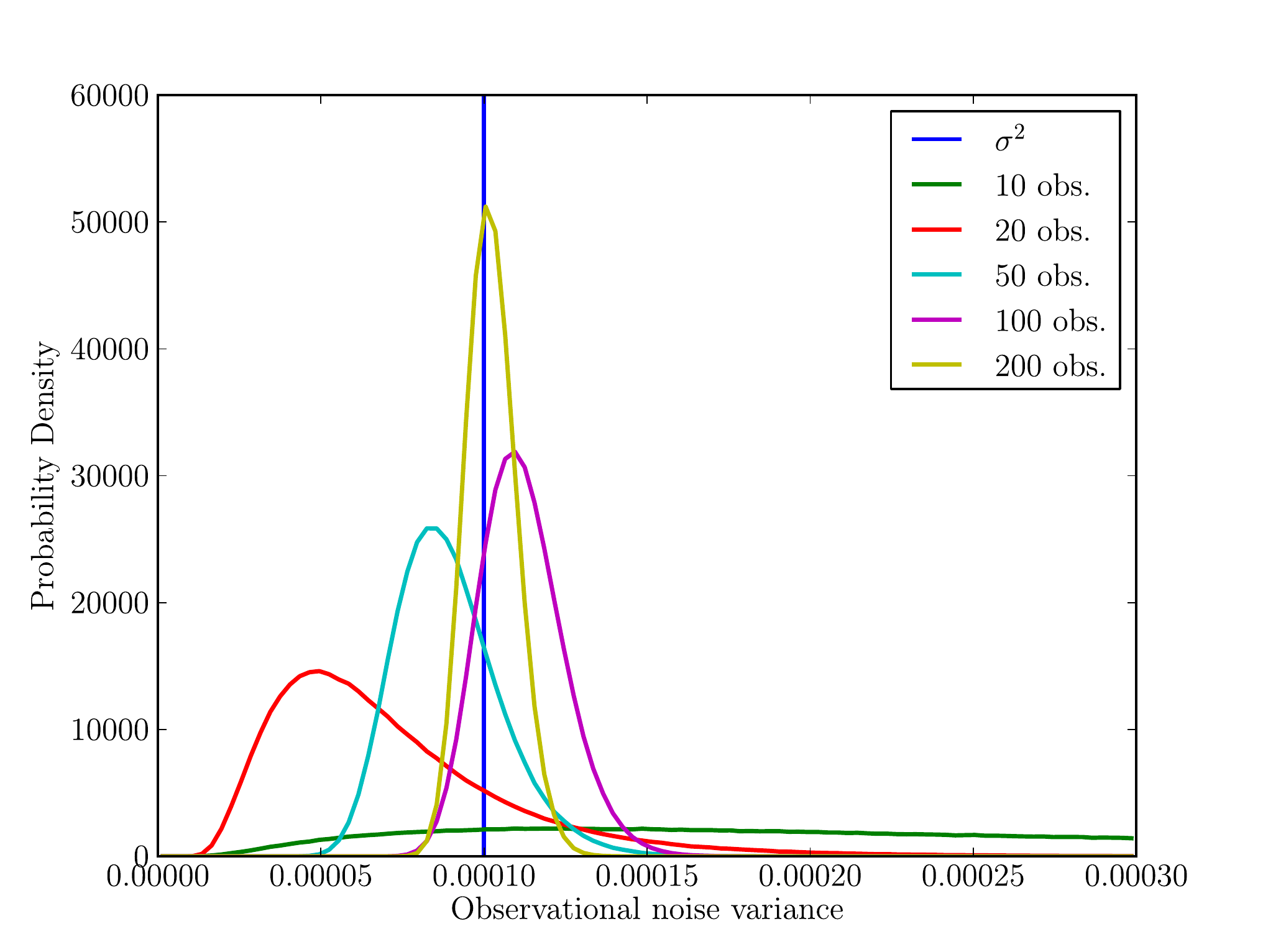}}} 
} 
\caption{\it (a) Plots of the deformed shapes using the mean values of
  $p_0$ from Figure \ref{NObs:means}. (b) Marginal distribution of
  observational noise variance in the MCMC
  chains.\label{NObs:shapes_var}}
\end{center} 
\end{figure}

However, we can look at the deformed shapes that are produced by the
mean initial momenta $\p_0$ shown in Figure \ref{NObs:means}. These
are shown in Figure \ref{NObs:shapes_var} (left). The marginal
distributions on the observational noise variance as estimated by the
MCMC methods are plotted in Figure \ref{NObs:shapes_var}
(right). These clearly show that as the number of observations
increases, the distribution is becoming increasingly peaked around the
value that was used when creating the data.
}

\subsection{Effect of lengthscale on multimodality}\label{sec:len}
One might ask what the advantage of full posterior sampling is over
other less computationally expensive optimisation approaches. In this
section we show that certain data scenarios can cause the posterior
distribution on the initial momentum $p_0$ to be complex with many
local maxima of probability density. In this case the solution from an
optimisation approach may not be unique, depending on the initial
state of the solver. Being able to characterise the whole of the
distribution in this case allows us to identify several different
possibilities for the function $p_0$, and to get a better idea of the
uncertainty in the problem. One such data scenario that exhibits this
posterior multimodality is where we have features in the data whose
lengthscale is smaller than the filtering lengthscale $\alpha$ in the
metric for the diffeomorphism, as we demonstrate in the following
example. Let us once again set the initial shape $\Gamma^1$ to be that
defined by the parameterisation $p^A(s)$ given in \eref{template}. We
then define $\Gamma^2(r)$ for $r \in [0,1]$, such that it consists of
a square of length 2, centred at $[\pi,\pi]$, with quarter-circles of
radius $r$ continuously added on each corner.

\begin{figure}[h]
\centering
\includegraphics[width=0.5\textwidth]{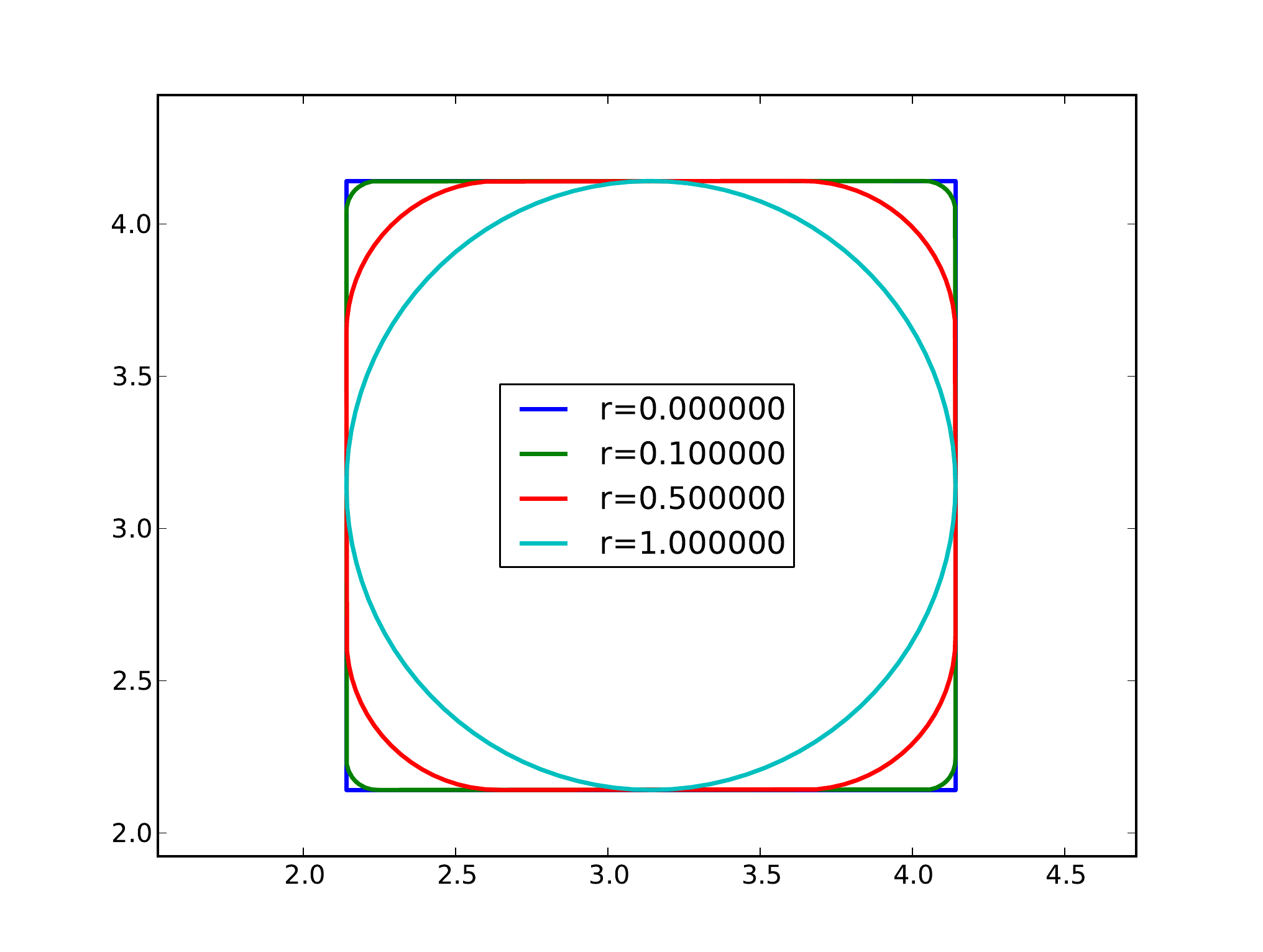}
\caption{\it Plot of the target shape with varying $r$, the radius of
  the quarter circles on the corners of the
  square. \label{square_circle_data}}
\end{figure}

Figure \ref{square_circle_data} shows the target shape for a range of
radii $r$. Note that for $r=0$, the shape is exactly a square of
length 2, and for $r=1$ the shape is a circle of radius 2, exactly
overlapping the template shape \eref{template}. The variable $r$
defines the minimum lengthscale of features in the data. Once $r$
becomes smaller than the maximum lengthscale than can be resolved by
the model, we would expect to see multimodality creeping into the
posterior distribution on $p_0$. Since the template shape cannot be
bent to imitate the features on such a small scale, near to this
feature only some of the observation points can be well-matched for
any given $p_0$. This problem does not cause multimodality in the
posterior distribution on the reparameterisation function $\nu$ since
the observation points on the template shape do not need to move in
order to minimise the distance between themselves and the observations
for any of the possible $p_0$ configurations in the multimodal
posterior on the initial momentum.

In the following experiments we look at what happens to the posterior
distribution on $(p_0,\nu)$ as $r$ starts at 1 and steadily
decreases. For each value of $r$, we consider the shape $\Gamma^2(r)$
as described above, and place $10^3$ observation points along the
shape, evenly spaced in terms of arc length. A large number of
observations are used so that the data can resolve small
lengthscales. We use these observation points as our data, with no
noise added, to isolate the cause of the multimodality purely to the
lengthscale of the features in the data. For the function $\Phi =
\frac{1}{2}\|\cG(p_0,\nu)-y\|^2_\Sigma$ in the likelihood functional,
we pick $\Sigma = \sigma^2I$ with $\sigma = 10^{-2}$.

\begin{figure}[htp] 
\begin{center} 
\mbox{ 
\subfigure[\label{LSP_0}]{\scalebox{0.31}{\includegraphics{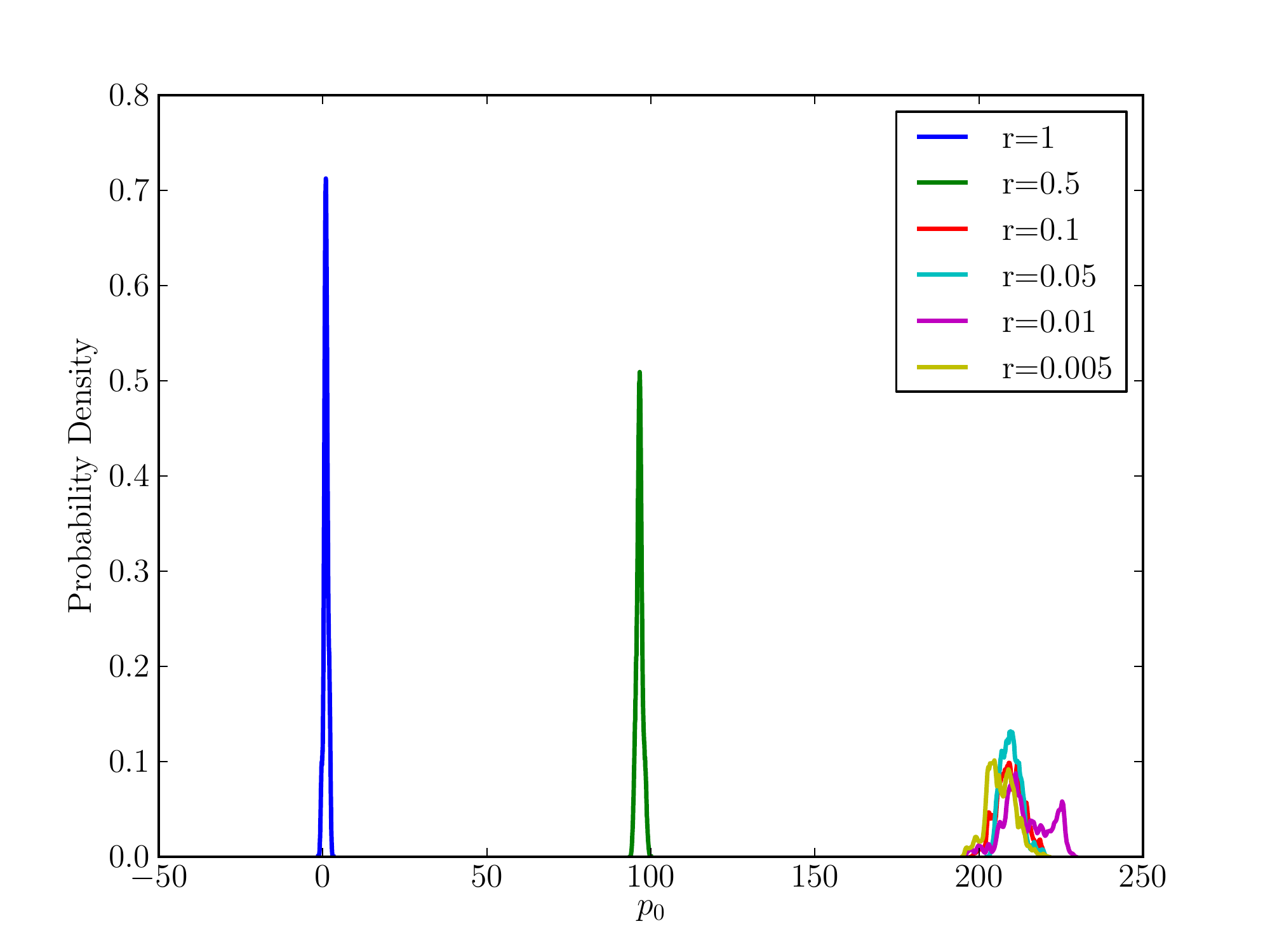}}} 
\subfigure[\label{LSP_1}]{\scalebox{0.31}{\includegraphics{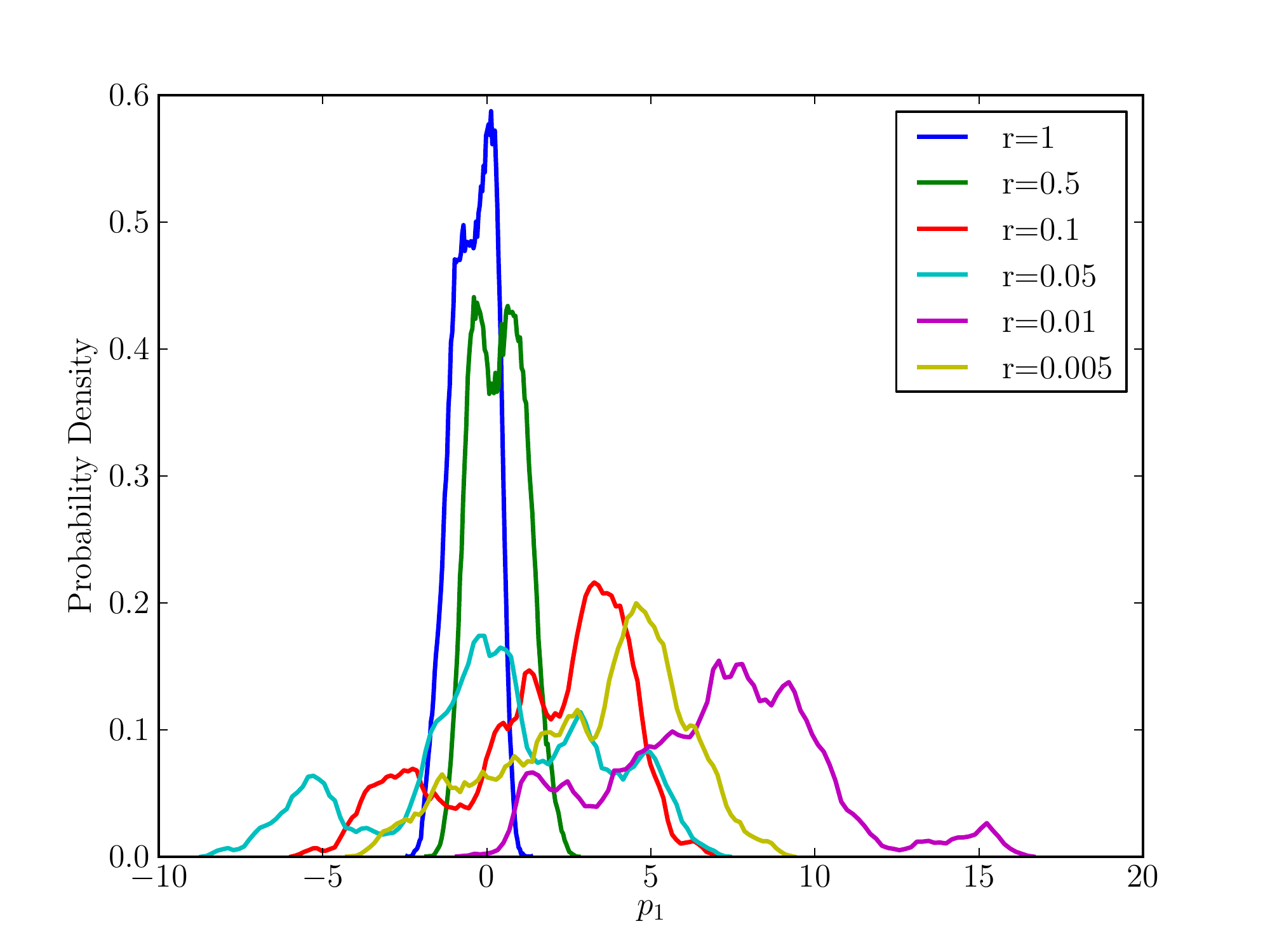}}} 
} 
\caption{\it Plot of marginal distributions on (a) the lowest, (b) the
  second lowest wave number Fourier mode of the initial momentum
  function $p_0$, given varying value of $r$, the radius of the quarter circles in the target shape.\label{LSP}}
\end{center} 
\end{figure} 

Figure \ref{LSP} shows the marginal distributions on the first two
Fourier modes of the initial momentum $p_0$ as the radius of the
quarter circles in the data is reduced. Notice that for $r=1$, where
the data is simply a circle, the distributions are smooth, monomodal
and bell-shaped. However, as $r$ is reduced, the distributions become
increasingly irregular, and in some cases multimodal. Moreover, these
later examples take a great deal of time to converge as the posterior
distribution is more complex{, and it is harder to be sure that we
  have adequately explored the whole of the probability
  measure. However, these results show that the posterior measures
  arising from this application are non-Gaussian, and not necessarily
  monomodal}. Simply finding the local maxima of probability density
for these examples would not be a good description of the nature of
the distribution as a whole, not would it help us in determining the
uncertainty in the problem.

\begin{figure}[htp] 
\begin{center} 
\mbox{ 
\subfigure[\label{LSNu_0}]{\scalebox{0.31}{\includegraphics{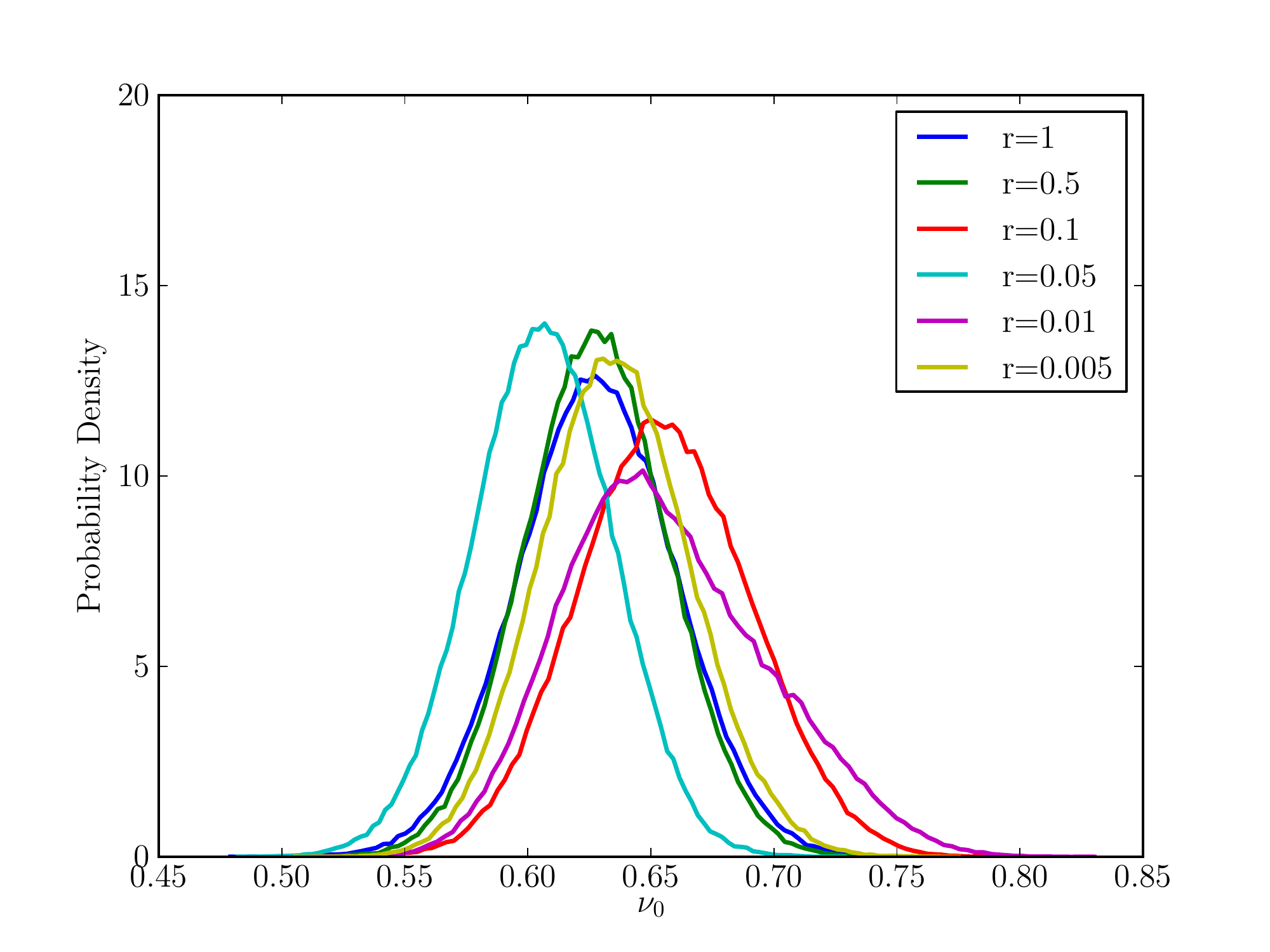}}} 
\subfigure[\label{LSNu_1}]{\scalebox{0.31}{\includegraphics{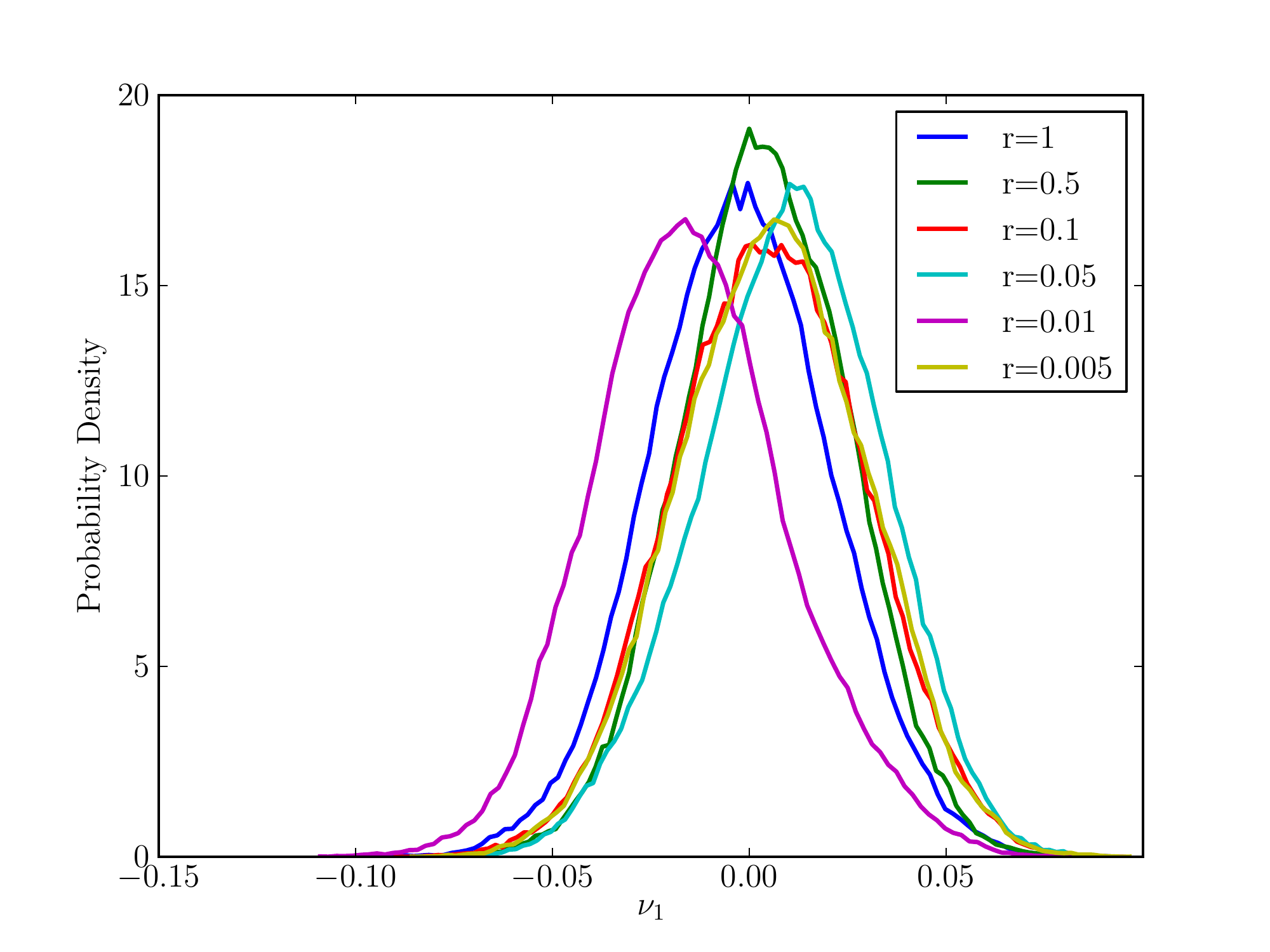}}} 
} 
\caption{\it Plot of marginal distributions on (a) the lowest, (b) the
  second lowest wave number Fourier mode of the reparameterisation
  function $\nu$, given varying value of $r$, the radius of the
  quarter circles in the target shape.\label{LSNu}}
\end{center} 
\end{figure} 
Similarly, Figure \ref{LSNu} shows the marginal distributions on the
first two Fourier modes of the reparameterisation function $\nu$ as
the radius of the quarter circles in the data is reduced. Note that
the distributions on $\nu$ are monomodal in nature, in contrast to
those on $p$ in this data scenario.

\subsection{Partial observations}\label{sec:par}
Another area of interest to many practitioners is algorithms which try
to recover information in the case where our observations of the
target shape are only partial. We may only observe part of the shape,
or our observations in some regions of the shape may be essentially
destroyed by excessive noise. In this case we would expect to find
many different possible shapes which could fit with our data equally
well, meaning that the uncertainty quantification through full
sampling of the Bayesian posterior is informative.

Let us consider the scenario where we have 1000 fairly evenly spaced
observations of the target shape, where the template shape is given by
\eref{template} and where the functions $(p_0, \nu)$ that create the
target shape and the positions of the observation points are draws
from the prior. Suppose now that the observation noise is no longer
i.i.d. around the entire shape, but instead we have two regions with
different covariance structure. In one region, covering approximately
three quarters of the shape's circumference, the noise is relatively
small, distributed as $\mathcal{N}(0,\Gamma)$ with $\Gamma =
\sigma_D^2 I$ with $\sigma_D = 0.0001$. In the final quarter of the
domain we add noise with much larger variance, with $\sigma_L =
0.1$. This value is so large that most of this data is rendered
useless, in terms of extrapolating information about the value of the
functions $p_0$ and $\nu$. {As in Section \ref{sec:PC}, a very
  fine resolution of the curve was used when simulating the data (1000
  points), and a coarser representation was used within the MCMC
  method (100 points).}

In this data scenario, we must also be careful in our choice of
functional $\Phi(p_0,\nu) = \frac{1}{2}\|\cG(p_0,\nu) -
y\|^2_{\Gamma_{\mbox{\tiny PO}}}$ in the likelihood. If we were to set
$\Gamma_{\mbox{\tiny PO}}=\sigma_D^2$ over the whole domain, then we
would be attempting to fit the shape very closely to the very noisy
observations in the final third, which would cause all manner of
problems. On the other hand, if we were to take $\Gamma_{\mbox{\tiny
    PO}}=\sigma_L^2$, we would be assuming a lot more uncertainty in
the majority of the observations than is actually the case. {In
  the numerics that follow, as we did in Section \ref{sec:PC}, we
  treat the observation noise variance $\sigma^2$ as an unknown also
  to be found.}

\begin{figure}[h]
\centering
\includegraphics[width=0.5\textwidth]{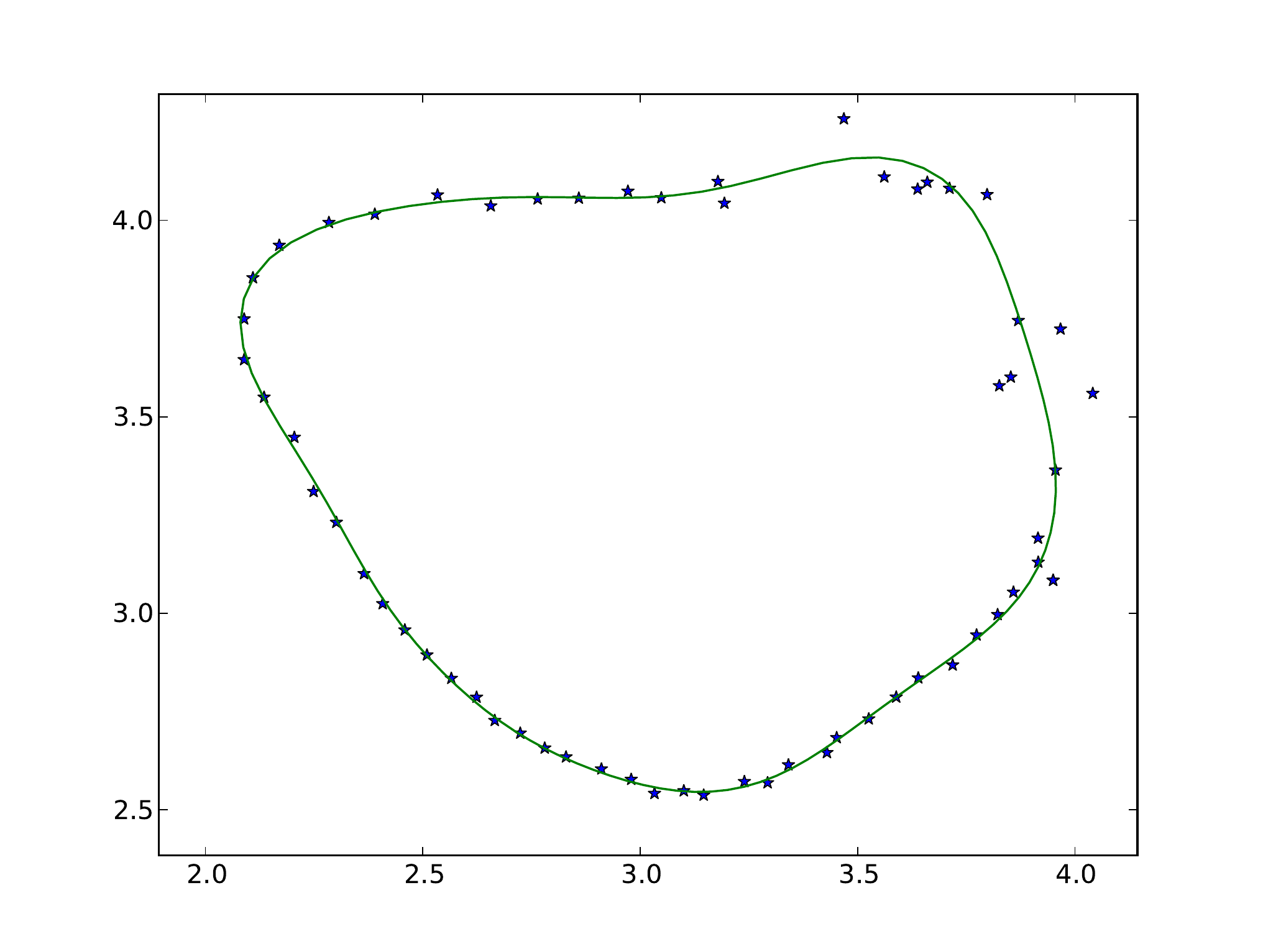}
\caption{\it Plot of partial data, due to large noise variance in one
  region of the observed shape. Green line denotes the converged
  solution of the corresponding optimisation
  problem. \label{partial_data}}
\end{figure}

Figure \ref{partial_data} shows an example of such a data
scenario. The stars denote data points, and the green curve denotes
the converged solution of the corresponding optimisation problem as
described in Section \ref{sec:RWMH}. However, this single curve does not
tell us the whole story, since there are many curves that could
satisfy the data (almost) as well as this curve.

{
\begin{figure}[htp] 
\begin{center} 
\mbox{ 
\subfigure[]{\scalebox{0.31}{\includegraphics{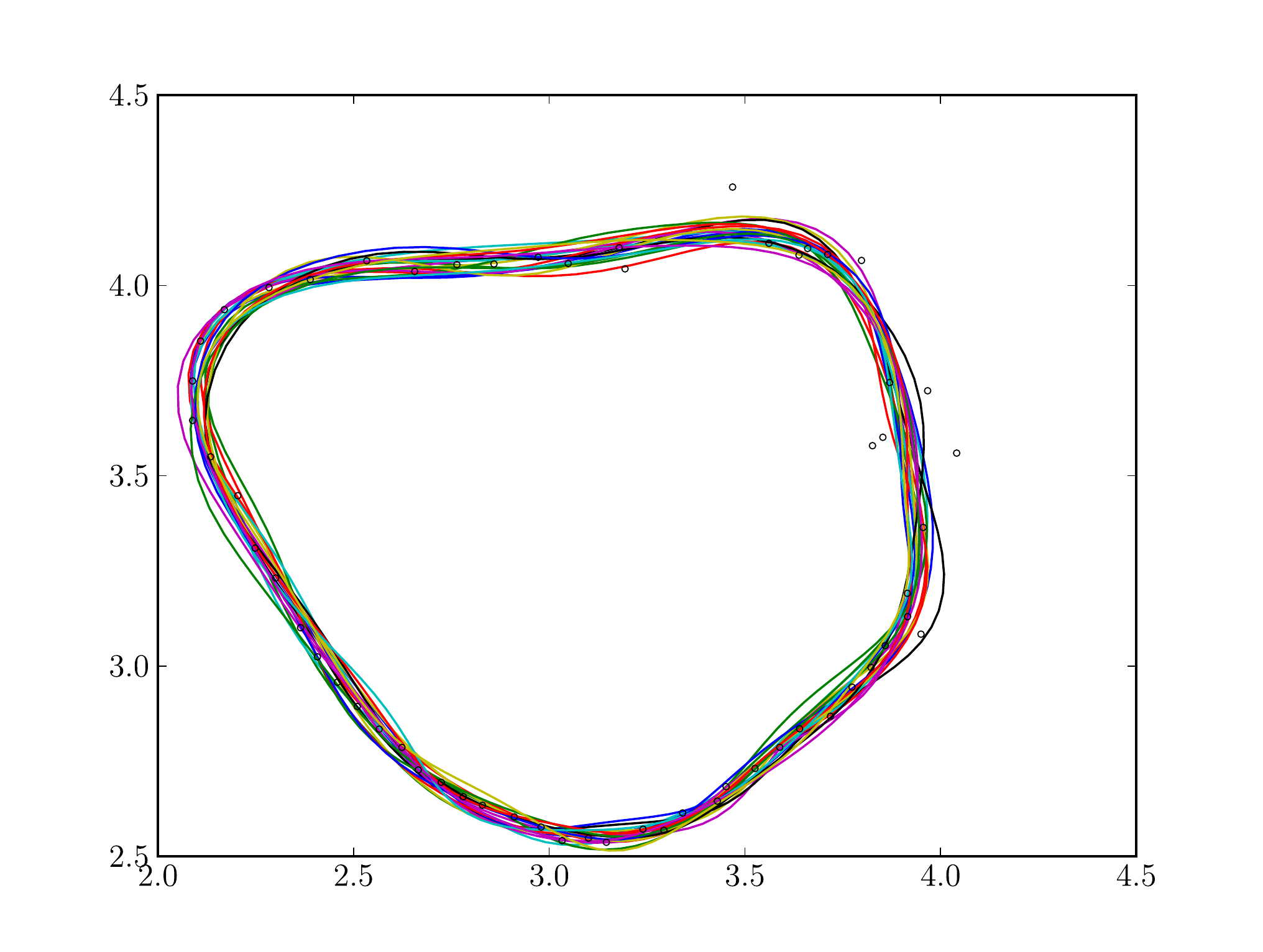}}} 
\subfigure[]{\scalebox{0.31}{\includegraphics{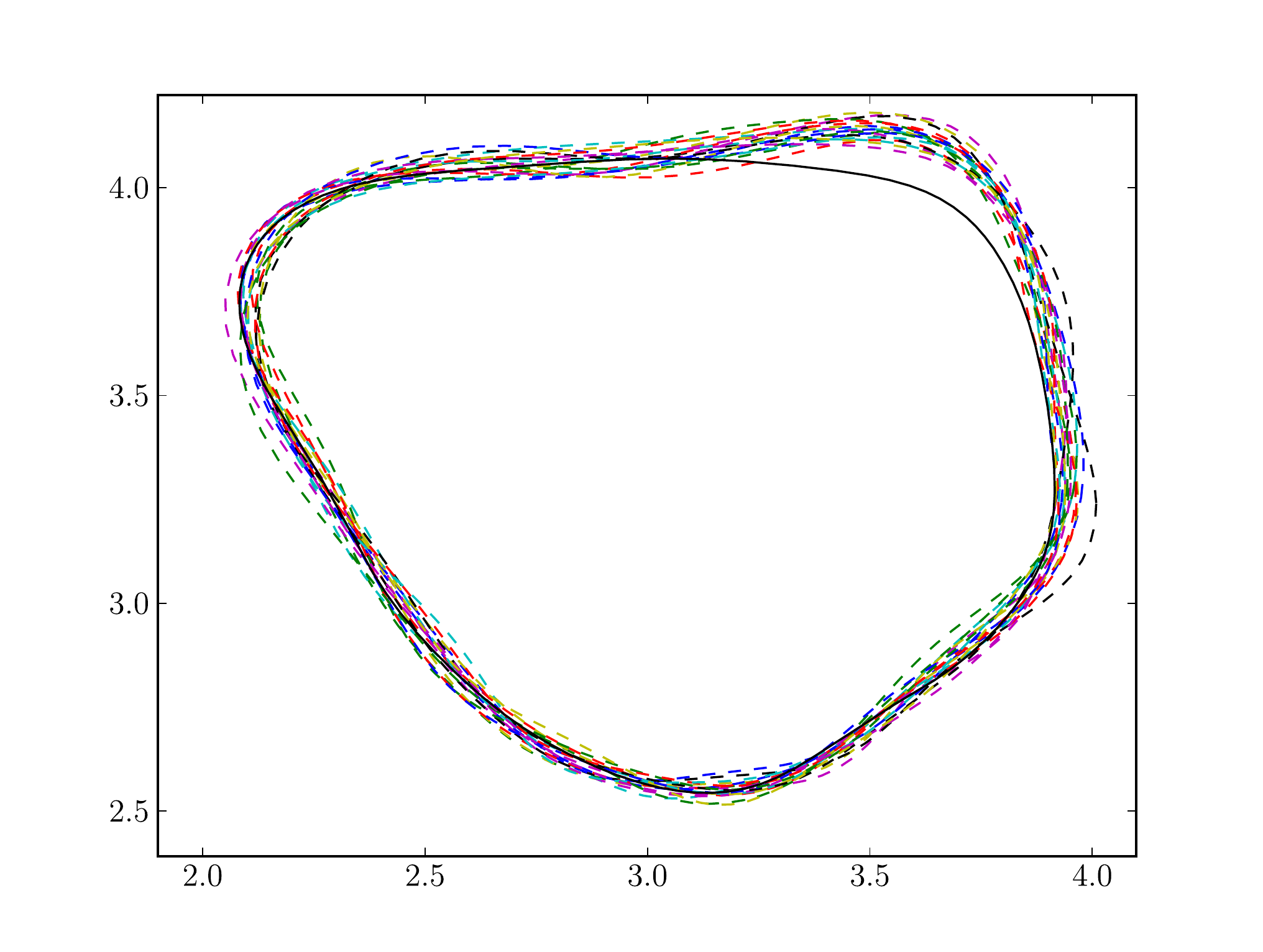}}} 
} 
\caption{\it Plot of 20 shapes arising from samples from the posterior
  distribution using the data plotted in \ref{partial_data} alongside
  (a) the data points, (b) the true shape (black curve). Since the
  recovered noise model assumes uniform observational noise over the
  whole shape, there is higher variability over the whole shape due to
  the extra noise in 1/4 of the data.\label{spag_plots}}
\end{center} 
\end{figure} 

Figure \ref{spag_plots} represents this variability through the plots
of 20 shapes arising from samples of the posterior distribution using
this data. The data is very poor in one region, and as such there are
several shapes which match the data equally as well. The data is so
poor in this one region that the recovery of the true shape in this
area is poor.

\begin{figure}[htp] 
\begin{center} 
\mbox{ 
\subfigure[\label{spag_data}]{\scalebox{0.31}{\includegraphics{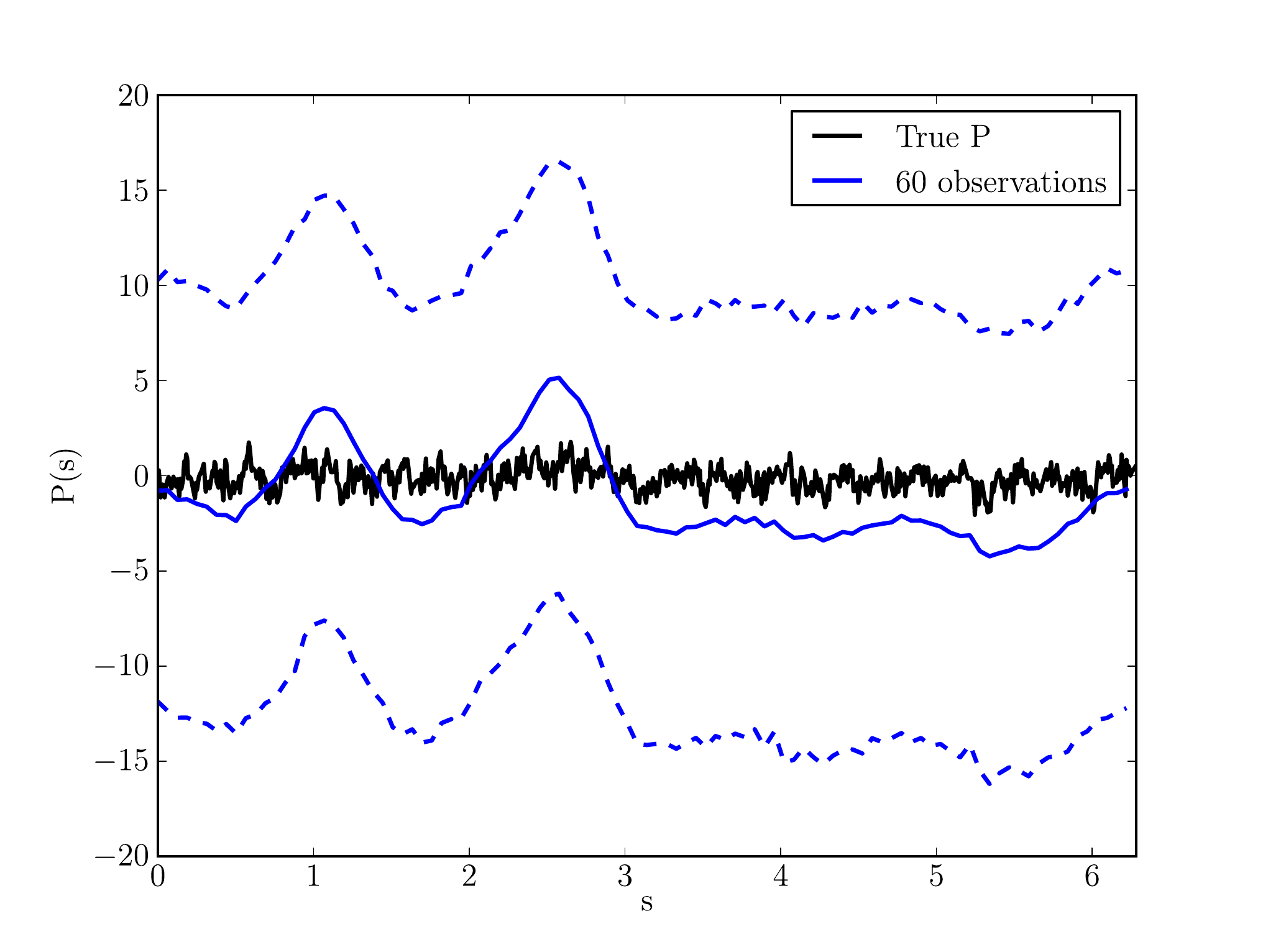}}} 
\subfigure[\label{spag_truth}]{\scalebox{0.31}{\includegraphics{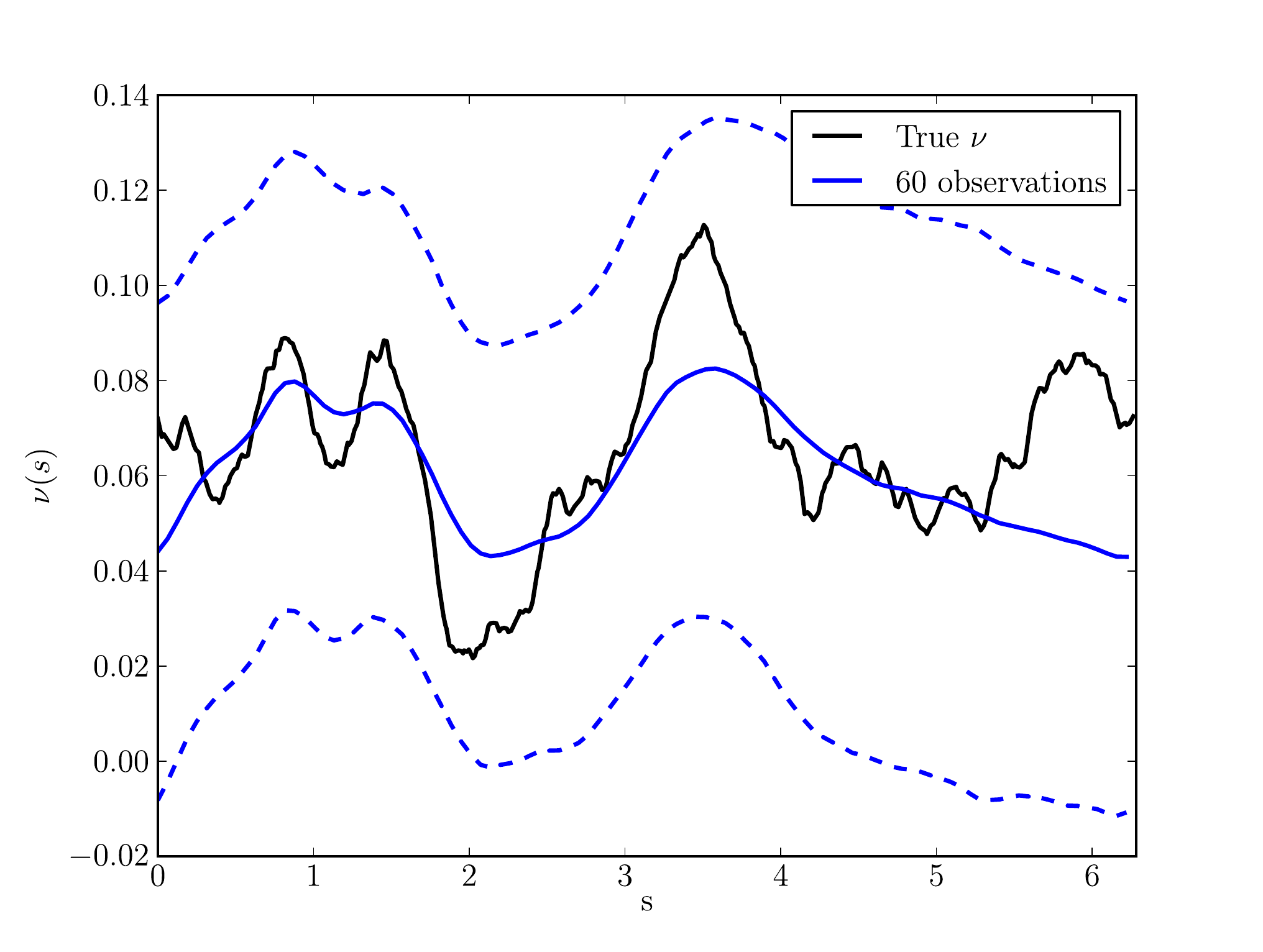}}} 
} 
\caption{\it Plots of the mean values (solid lines) of (a) $p_0$ and $\nu$ occurring in the posterior distribution using the data plotted in \ref{partial_data}. The dashed lines denote $\pm 3$ standard deviations (approximately 99.7\% confidence interval)\label{partial_means}}
\end{center} 
\end{figure} 

This variability can also be seen in Figure \ref{partial_means}, where
the mean and variance of the functions $p_0$ and $\nu$ in posterior
distribution are visualised.

\begin{figure}[h]
\centering
\includegraphics[width=0.5\textwidth]{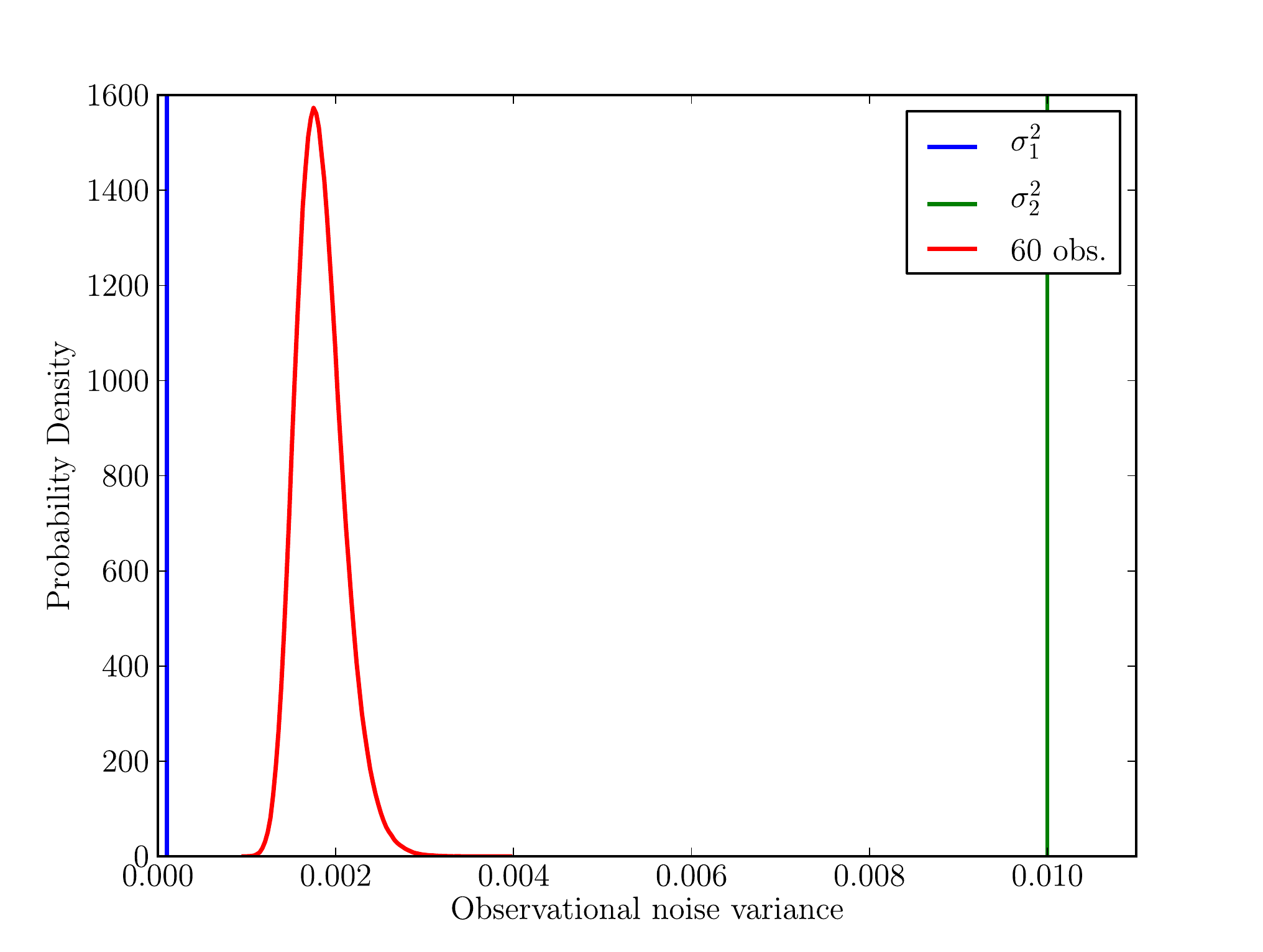}
\caption{\it Marginal distribution of observational noise variance in
  the MCMC chain, using data plotted in \ref{partial_data}. The two
  different variances used when creating the data are also
  shown.\label{par_var}}
\end{figure}

Figure shows the marginal distributions on the observational noise
variance as estimated by the MCMC methods. The non-uniformity of the
noise used when creating the data cannot be recovered in this setting
since we assume uniformity over the whole shape in our prior on the
observational noise variance. Therefore, the recovered values are
somewhere between the poor and good SNR regimes.  } Note that in
comparison with the distributions in Figure \ref{NObs:means}, the
variances in these distributions are quite similar, despite having a
great deal more observations. With a region with a much poorer SNR,
the uncertainty in the system is much greater, and as such there is
still a large region of the state space with similar probability
density. The approach we have used in this paper allows us to quantify
this uncertainty as well as finding states of maximal likelihood.
\section{Conclusions}\label{sec:conc}
This paper presents a Bayesian approach to shape registration which
takes into account the fact that observations of a shape are not
exact, and gives a distribution on the shortest distance in shape
space between the template shape to the observed shape. We concentrate
on the case of a finite number of observations of points from the
target curve, and our approach makes use of the use of explicit
reparameterisation variable as described in
\cite{CoHo2010,CoClPe2011}. The sampling method is framed on function
space and so is robust under refinements of discretisation
\cite{crsw11}.

By careful analysis of the forward problem, we have been able to
formulate a well-posed Bayesian inverse problem regarding shape
matching of a curve with a set of noisy observations of another. We
have shown that the likelihood function is Lipschitz continuous on a
space which has full measure with respect to a specified choice of
Gaussian prior measure. Using this, we have shown how to draw samples
from well-defined posterior distributions using the \SLC{pCN} MCMC
sampler on function space. This choice of algorithm prevents slower
convergence of the statistical algorithm as the discretisation of the
{observation operator} is refined. We have then implemented this
algorithm, and presented briefly some illustrative numerics.

In these numerics, we first showed that the posterior distribution
shows consistency as the number of observations of a recoverable shape
are increased, and the posterior distributions on the functions of
interest become increasingly peaked on the functions which were used
to create the data. We also showed an example, where the lengthscale
of features in the data is small, in which the posterior distributions
exhibit multimodal behaviour, indicating that full characterisation of
the posterior can give us more information than the solution of an
equivalent optimisation problem. Since the shape space is a Riemannian
manifold with regions of positive and negative curvature, there may be
more than one geodesic between two points in shape space, which may
also lead to multimodality.

Finally we showed that this method
can give us a range of different possible solutions in the case of
lost or destroyed data for some part of the target shape.

Since we already have at our disposal an implementation of the adjoint
problem so that we can calculate the gradient of the observation
operator (which we use in the BFGS method for the deterministic
burn-in), we can very simply adapt the MCMC method from random walk to
a gradient method, such as the Metropolis Adjusted Langevin Algorithm
(MALA). This would increase the efficiency of the algorithm
markedly. Further analytical results would be needed, however, to
ensure that the gradient of the observation operator is continuous on
a space which has full measure with respect to an appropriately chosen
prior distribution.

Another extension of this work would be to also make the problem
translation, rotation and scale invariant, so that any misalignment of
the imaging equipment has a negligible effect on the results. This
would involve adding parameters into the state space to allow for
these types of operations. One might also consider the case where the
template shapes themselves are only noisily observed, so that we are
trying to find a distribution on the length of geodesic paths in shape
space between two noisily observed shapes. The problem could also be
extended to the case in which the ordering of the observed points
around the curve is not known, in which case the discrete ordering
would also have to be learned as part of the inverse problem. This
could be done by combining the present approach with a Gibb's sampler
for the ordering. This would result in a Bayesian approach to the
segmentation process, in which points on the boundary between two
materials might first be estimated from a bitmapped image but their
ordering is not known with 100\% confidence. The result would be a
probability distribution on the boundary curve with some prescribed
topology.

\bibliography{MCMC_shape}

\begin{thebibliography}{10}

\bibitem{MiTrYo2003}
M.~I. Miller, A.~Trouv\'e, and L.~Younes.
\newblock On the metrics, euler equations and normal geodesic image motions of
  computational anatomy.
\newblock In {\em Proceedings of the 2003 International Conference on Image
  Processing}, volume~2, pages 635--638. IEEE, 2003.

\bibitem{MiTrYo2006}
M.~I. Miller, A.~Trouv\'e, and L.~Younes.
\newblock Geodesic shooting for computational anatomy.
\newblock {\em Journal of Mathematical Imaging and Vision}, 24(2):209--222,
  2006.

\bibitem{VaMiYoTr2004}
M.~Vailliant, M.~I. Miller, L.~Younes, and A.~Trouv\'e.
\newblock Statistics on diffeomorphisms via tangent space representations.
\newblock {\em NeuroImage}, 13:161--169, 2004.

\bibitem{FlLuPiJo2004}
P.T. Fletcher, C.~Lu, S.~M. Pizer, and S.~Joshi.
\newblock Principal geodesic analysis for the study of nonlinear statistics of
  shape.
\newblock {\em IEEE Transactions on Medical Imaging}, 23(8):995--1005, 2004.

\bibitem{AlKuTr2010}
S.~Allassonnière, Estelle Kuhn, , and Alain Trouvé.
\newblock Construction of bayesian deformable models via stochastic
  approximation algorithm: A convergence study.
\newblock {\em Bernoulli}, 16(3):641--678, 2010.

\bibitem{AlKuTr2008}
S.~Allassonnière, Estelle Kuhn, and Alain Trouvé.
\newblock Map estimation of statistical deformable template via nonlinear mixed
  effect models : Deterministic and stochastic approaches.
\newblock {\em Proc. of Mathematical Foundations of Computational Anatomy},
  2008.

\bibitem{Ma_etal2008}
J.~Ma, M.I. Miller, A.~Trouvé, and L.~Younes.
\newblock Bayesian template estimation in computational anatomy.
\newblock {\em NeuroImage}, 42(1):256--261, 2008.

\bibitem{HoRaTrYo2004}
D.~D. Holm, J.~T. Ratnanather, A.~Trouvé, and L.~Younes.
\newblock Soliton dynamics in computational anatomy.
\newblock {\em NeuroImage}, 23, Supplement 1:S170 -- S178, 2004.

\bibitem{Glaunes}
J.~Glaunes, A.~Trouv\'e, and L.~Younes.
\newblock Diffeomorphic matching of distributions: A new approach for
  unlabelled point-sets and sub manifolds matching.
\newblock Proceedings of the 2004 IEEE Computer Society Conference on Computer
  Vision and Pattern Recognition.

\bibitem{VaGl2005}
Marc Vaillant and Joan Glaunes.
\newblock Surface matching via currents.
\newblock In {\em IPMI}, pages 381--392, 2005.

\bibitem{CoHo2010}
C.J. Cotter and D.D. Holm.
\newblock Geodesic boundary value problems with symmetry.
\newblock {\em J. Geom. Mech.}, 2:51--68, 2010.

\bibitem{Stuart_Bayes}
A.~M. Stuart.
\newblock Inverse problems: A bayesian perspective.
\newblock {\em Acta Numerica}, 19:451--559, 2010.

\bibitem{cdrs08}
S.L. Cotter, M.~Dashti, J.C. Robinson, and A.M. Stuart.
\newblock Bayesian inverse problems for functions and applications to fluid
  mechanics.
\newblock {\em Inverse Problems}, 25, 2009.

\bibitem{cds10}
S.~L. Cotter, M.~Dashti, and A.~M. Stuart.
\newblock Approximation of {B}ayesian inverse problems for {PDE}s.
\newblock {\em SIAM J. Numer. Anal.}, 48(1):322--345, 2010.

\bibitem{cds11}
S.~L. Cotter, M.~Dashti, and A.~M. Stuart.
\newblock Variational data assimilation using targetted random walks.
\newblock {\em International Journal for Numerical Methods in Fluids}, pages
  n/a--n/a, 2011.

\bibitem{crsw11}
S.L. Cotter, G.O. Roberts, A.M. Stuart, and D.~White.
\newblock {MCMC} methods for functions: modifying old algorithms to make them
  faster.
\newblock {\em Submitted to Statistical Science}, 2012.

\bibitem{MiYo2001}
M.~T. Miller and L.~Younes.
\newblock Group actions, homeomorphisms, and matching: A general framework.
\newblock {\em International Journal of Computer Vision}, 41:61--84, 2001.

\bibitem{GlTrYo04}
J.~Glaunes, A.~Trouv{\'e}, and L.~Younes.
\newblock Diffeomorphic matching of distributions: {A} new approach for
  unlabelled point-sets and sub-manifolds matching.
\newblock In {\em IEEE Computer Society Conference on Computer Vision and
  Pattern Recognition}, volume~2, pages 712--718, 2004.

\bibitem{GaBaRa2011}
F.~Gay-Balmaz and T.S. Ratiu.
\newblock Clebsch optimal control in formulation in mechanics.
\newblock {\em J. Geom. Mech.}, 3(1), 2011.

\bibitem{CoClPe2011}
C.J. Cotter, A.~Clark, and J~Peir\'o.
\newblock A reparameterisation-based approach to geodesic shooting for curve
  matching.
\newblock {\em Int. J. Comp. Vis.}, pages 1--19, 2012.
\newblock doi:10.1007/s11263-012-0520-0.

\bibitem{KeWe2008}
B.~Khesin and R.~Wendt.
\newblock {\em Geometry of infinite-dimensional groups}, volume~51 of {\em
  Ergebnisse der Mathematik und Grenzgebiete 3.Folge}, chapter~1.
\newblock Springer-Verlag, 2008.

\bibitem{broyden70}
C.G. Broyden.
\newblock The convergence of a class of double-rank minimization algorithms.
\newblock {\em J. Inst. Maths Applics}, 6:76--90, 1970.

\bibitem{shanno85}
D.F. Shanno.
\newblock On {B}royden-{F}letcher-{G}oldfarb-{S}hanno method.
\newblock {\em J. Optimiz. Theory Appl.}, 46, 1985.

\bibitem{Co2008}
C.~J. Cotter.
\newblock The variational particle-mesh method for matching curves.
\newblock {\em J. Phys. A: Math. Theor.}, 41:344003, 2008.

\bibitem{Agapiou_post_con}
S.~Agapiou, S.~Larsson, and A.~M. Stuart.
\newblock Posterior consistency of the bayesian approach to linear ill-posed
  inverse problems.
\newblock arXiv:1203.5753v2.

\bibitem{Knapik_BIPWGP}
B.~T. Knapik, A.W. van~der Vaart, and J.~H. van Zanten.
\newblock {Bayesian inverse problems with Gaussian priors}.
\newblock {\em {Annals of Statistics}}, {39}({5}):{2626--2657}, {Oct} {2011}.

\bibitem{knapik_heat_eq}
B.~T. Knapik, A.~W. van~der Vaart, and J.H. van Zanten.
\newblock Bayesian recovery of the initial condition for the heat equation.
\newblock {\em Comm. Statist. Theory Methods}, To appear.

\bibitem{JKES}
J.~Kaipio and E.~Somersalo.
\newblock {\em Statistical and Computational Inverse Problems}, chapter~2.
\newblock Springer, 2004.

\end{thebibliography}

\noindent {\bf Acknowledgements}
The research leading to these results has received funding from the
European Research Council under the {\it European Community's} Seventh
Framework Programme ({\it FP7/2007-2013})/ ERC {\it grant agreement}
No. 239870.  This publication was based on work supported in part by
Award No KUK-C1-013-04, made by King Abdullah University of Science
and Technology (KAUST), and the results were obtained using the
Imperial College High Performance Computing Centre cluster. SLC would
also like to thank St Cross College Oxford for support via a Junior
Research Fellowship.
\end{document}